\documentclass[12pt]{amsart} 
\usepackage{a4wide} 
\usepackage[utf8]{inputenc} 
\usepackage[english]{babel} 
\usepackage{fourier} 
\usepackage{dsfont} 
\usepackage{url}

\theoremstyle{plain} 
\newtheorem{thm}{Theorem} 
\newtheorem{lem}[thm]{Lemma} 
\newtheorem{cor}[thm]{Corollary}

\newtheorem*{lemun}{Lemma}

\theoremstyle{definition} 
\newtheorem*{defn}{Definition} 
\newtheorem*{ex}{Example}

\theoremstyle{remark} 
\newtheorem*{rem}{Remark} 
\newtheorem*{que}{Question}

\def\DD{\mathbb D} 
\def\CC{\mathbb C} 
\def\RR{\mathbb R} 
\def\TT{\mathbb T} 
\def\NN{\mathbb N} 
\def\QQ{\mathbb Q}
\def\veps{\varepsilon}

\begin{document} 

\title{Compact composition operators with non-linear symbols on the $H^2$ space of Dirichlet series} 
\date{\today} 

\author{Fr\'{e}d\'{e}ric Bayart} \address{Clermont Universit\'{e}, Université Blaise Pascal,
Laboratoire de Mathématiques,
UMR 6620 - CNRS,
Campus des Cézeaux,
3, place Vasarely,
TSA 60026,
CS 60026,
63178 Aubière cedex} \email{bayart@math.univ-bpclermont.fr}

\author{Ole Fredrik Brevig} \address{Department of Mathematical Sciences, Norwegian University of Science and Technology (NTNU), NO-7491 Trondheim, Norway} \email{ole.brevig@math.ntnu.no}

\thanks{The second author is supported by Grant 227768 of the Research Council of Norway.}

\subjclass[2010]{Primary 47B33. Secondary 30B50, 30H10.}
\begin{abstract}
	We investigate the compactness of composition operators on the Hardy space of Dirichlet series induced by a map $\varphi(s)=c_0s+\varphi_0(s)$, where $\varphi_0$ is a Dirichlet polynomial. Our results depend heavily on the characteristic $c_0$ of $\varphi$ and, when $c_0=0$, on both the degree of $\varphi_0$ and its local behaviour near a boundary point. We also study the approximation numbers for some of these operators. Our methods involve geometric estimates of Carleson measures and tools from differential geometry. 
\end{abstract}

\maketitle

\section{Introduction} A theorem of Gordon and Hedenmalm \cite{GH99} describes the bounded composition operators on the Hilbert space $\mathcal{H}^2$ of Dirichlet series,
\[f(s) = \sum_{n=1}^\infty a_n n^{-s},\]
with square summable coefficients endowed with the norm $\|f\|_{\mathcal H^2}^2:=\sum_{n=1}^\infty|a_n|^2$. We let $\mathbb{C}_\theta$ denote the half-plane of complex numbers $s = \sigma+it$ with $\sigma>\theta$. The Dirichlet series in $\mathcal{H}^2$ represent analytic functions in $\mathbb{C}_{1/2}$ and a mapping $\varphi$ of $\mathbb{C}_{1/2}$ into itself defines a function $\mathcal C_\varphi(f):=f\circ\varphi$ on $\CC_{1/2}$, if $f \in \mathcal{H}^2$. The operator $\mathcal{C}_\varphi\,\colon\,\mathcal{H}^2\to\mathcal{H}^2$ is well-defined and bounded if and only if $\varphi$ is a member of the following class: 
\begin{defn}
	The \emph{Gordon--Hedenmalm class}, denoted $\mathcal{G}$, is the set of functions $\varphi\colon\mathbb{C}_{1/2}\to\mathbb{C}_{1/2}$ of the form 
	\begin{equation}
		\label{eq:compclass} \varphi(s) = c_0 s + \sum_{n=1}^\infty c_n n^{-s} =: c_0s + \varphi_0(s), 
	\end{equation}
	where $c_0$ is a non-negative integer called the \emph{characteristic} of $\varphi$, the Dirichlet series $\varphi_0$ converges uniformly in $\mathbb{C}_\varepsilon$ $(\varepsilon>0)$ and has the following mapping properties:
	
	{\normalfont(a)} If $c_0=0$, then $\varphi_0(\mathbb{C}_0)\subset \mathbb{C}_{1/2}$.
	
	{\normalfont(b)} If $c_0\geq1$, then either $\varphi_0\equiv 0$ or $\varphi_0(\mathbb{C}_0)\subset \mathbb{C}_0$. 
\end{defn}

Since the paper of Gordon and Hedenmalm, several authors have studied the properties of composition operators acting on $\mathcal H^2$ or on similar spaces of Dirichlet series (see for instance \cite{BAYMONAT,Ba03,FQ04,FQV04,QS14}). In the present work, we are interested in the study of the compactness of $\mathcal C_\varphi$ when $\varphi$ is a polynomial symbol, say 
\begin{equation}
	\label{eq:diripoly} \varphi(s) = c_0s + c_1 + \sum_{n=2}^N c_n n^{-s}, 
\end{equation}
and we implicitly assume that $\varphi \in \mathcal{G}$. The symbol $\varphi$ is said to have \emph{unrestricted range} if
\[\inf_{s \in \mathbb{C}_0} \mathrm{Re}\left(\varphi(s)\right) = 
\begin{cases}
	1/2 & \text{ if } c_0 = 0, \\
	0 & \text{ if } c_0 \geq 1. 
\end{cases}
\]
Correspondingly, if $\varphi(\mathbb{C}_0)$ is strictly contained in any smaller half-plane, we say that $\mathcal{C}_\varphi$ has \emph{restricted range}. It is well-known that the composition operator $\mathcal C_{\varphi}$ is compact when $\varphi$ has restricted range \cite[Thm.~21]{BAYMONAT}. In what follows, we will assume that $\varphi$ has unrestricted range.
\begin{defn}
	A set of integers $\Lambda \subseteq \mathbb{N}-\{1\}$ is called $\mathbb{Q}$\emph{-independent} if the set $\left\{\log{n}\,:\, n \in \Lambda\right\}$ is linearly independent over $\mathbb{Q}$. 
\end{defn}

Symbols of the form \eqref{eq:diripoly} have been extensively studied in the \emph{linear case}, 
\begin{equation}
	\label{eq:linearcase} \varphi(s) = c_0s + c_1 + \sum_{j=1}^d c_{q_j} q_j^{-s}, 
\end{equation}
where the set $\{q_j\}$ is $\mathbb{Q}$-independent and $c_{q_j}\neq0$. When $c_0 \geq 1$, it is proven in \cite{Ba03} that the operator $\mathcal C_{\varphi}$ is compact if and only if $\varphi$ has restricted range. Our first result extends this to the case of an arbitrary polynomial: 
\begin{thm}
	\label{thm:c0} Let $\varphi$ be a Dirichlet polynomial of the form \eqref{eq:diripoly} with $c_0 \geq 1$. Then $\mathcal C_{\varphi}$ is compact if and only if $\varphi$ has restricted range. 
\end{thm}

As is to be expected when investigating composition operators on $\mathcal{H}^2$, the symbols with $c_0=0$ are more difficult to handle and require different techniques. In this case, it is proven independently in \cite{Ba03} and \cite{FQV04} that composition operators induced by linear symbols \eqref{eq:linearcase} with $c_0=0$ are compact if and only if $\varphi$ has restricted range or $d\geq2$.

The main effort of this paper is dedicated to extending this result to general polynomials. We rely crucially on a geometric description of such compact composition operators found in \cite{QS14} (see Lemma~\ref{lem:QS} below). Our second result is: 
\begin{thm}
	\label{thm:polynomials} Suppose that $\{q_j\}_{j=1}^d$ are $\mathbb{Q}$-independent and that
	\[\varphi(s) = \sum_{j=1}^d P_j(q_j^{-s})\]
	is in $\mathcal{G}$, and that the polynomials $P_j$ are non-constant. Then $\mathcal C_\varphi$ is compact if and only if $\varphi$ has restricted range or $d\geq2$. 
\end{thm}

Theorem~\ref{thm:polynomials} is truly a non-linear extension of the results for linear symbols, however it fails to handle the relatively simple cases 
\begin{equation}
	\label{eq:simplemixedcase} \varphi_1(s) = \frac 92 - 2^{-s} - 3^{-s} - 2\cdot 6^{-s}\qquad\text{ and }\qquad\varphi_2(s)=\frac{13}2-4\cdot 2^{-s}-4\cdot 3^{-s}+2\cdot 6^{-s}, 
\end{equation}
where ``mixed terms'' are present. However, the compactness of the associated operators can be decided by our main result. Before this result can be stated, we need to introduce some additional definitions.
\begin{defn}
	Let $\Lambda \subseteq \mathbb{N}-\{1\}$. We let the \emph{complex dimension} of $\Lambda$, denoted $\mathcal{D}(\Lambda)$, be the infimum of $\mathrm{card}(\Lambda_0)$ where $\Lambda_0\subset\NN-\{1\}$ is $\mathbb{Q}$-independent and multiplicatively generates $\Lambda$. 
\end{defn}

At this point, we should mention that the set $\Lambda_0$ attaining such an infimum is not necessarily unique. This is easily seen by considering $\Lambda = \left\{2^2\cdot3^2,\,2^4\cdot3^2,\,2^2\cdot3^4,\,2^4\cdot3^4\right\}$, where $\Lambda_0$ can be chosen as any of the following sets:
\[\left\{2,\,3\right\}\qquad\left\{2^2,\,3\right\}\qquad\left\{2,\,3^2\right\}\qquad\left\{2^2,\,3^2\right\}\qquad\left\{2^2\cdot3,\,3\right\}\qquad\left\{2,\,2\cdot 3^2\right\}\]
Now, we will rewrite \eqref{eq:diripoly} as
\begin{equation}\label{eq:dirpoly2}
\varphi(s) = c_1 + \sum_{n \in \Lambda} c_n n^{-s}
\end{equation}
with $c_n\neq 0$ for every $n\in \Lambda$. We pick some $\Lambda_0 = \{q_1,\,q_2,\,\ldots,\,q_d\}$ where $d = \mathcal{D}(\Lambda)$. Since $\Lambda_0$ generates $\Lambda$, any $n \in \Lambda$ can be written uniquely as a product of elements in $\Lambda_0$,
\[n = \prod_{j=1}^d q_j^{\alpha_j}.\]
This associates to $n$ the $d$-dimensional multi-index $\alpha(n)$. Clearly, $\alpha(n)$ depends on the choice of $\Lambda_0$ as the example considered above illustrates.
\begin{defn}
	The \emph{degree of} $\varphi$ \emph{with respect to} $\Lambda_0$ is defined by
	\[\deg(\varphi,\Lambda_0)=\sup\big\{|\alpha(n)|=\alpha_1+\alpha_2+\cdots+\alpha_d \,:\, n \in \Lambda\big\}.\]
	Among the different $\Lambda_0$ which generate $\Lambda$ and with $\mathrm{card}(\Lambda_0)=\mathcal D(\Lambda)$, we choose an optimal $\Lambda_0$ in the sense that it minimizes $\deg(\varphi,\Lambda_0)$. The \emph{degree} of $\varphi$ is then equal to the value of $\deg(\varphi,\Lambda_0)$ where $\Lambda_0$ is optimal in the previous sense. 
\end{defn}

It is clear that there can be more than one optimal $\Lambda_0$, as the example considered above again demonstrates, where the three final possibilities all have $\deg(\varphi,\Lambda_0)=4$ if $\varphi$ is given by \eqref{eq:dirpoly2}.
\begin{rem}
	For maps of the form \eqref{eq:linearcase} as considered before, the complex dimension is equal to $d$ and the degree is equal to 1, which justifies our terminology ``linear case''. 
\end{rem}

The study of the Hardy space of Dirichlet series $\mathcal H^2$ is intimately related to function theory on polydiscs. In our concerns, the main tool will be the so-called Bohr lift. Indeed, consider an optimal $\Lambda_0$ and use the substitution $q_j^{-s} \mapsto z_j$. To simplify the expressions in what follows, we will also subtract 1/2. Hence we obtain a polynomial in $d$ variables with the same degree as $\varphi$, 
\begin{equation}
	\label{eq:bohrlift} \Phi(z) =\left( c_1 -\frac 12\right)+ \sum_{n \in \Lambda} c_n z^{\alpha(n)}. 
\end{equation}
The polynomial $\Phi$ will be called an \emph{optimal Bohr lift} of $\varphi$. Using Kronecker's theorem (see for instance \cite[Ch.~13]{hardywright}), the $\QQ$-independence of $\Lambda_0$ implies that $\Phi$ maps $\mathbb{D}^d$ into $\mathbb{C}_0$. The polynomial $\Phi$ induces a map, denoted by $\phi$, on $\mathbb{R}^d$ defined by
\[\phi(\theta_1,\,\theta_2,\,\ldots,\,\theta_d)=\Phi\big(e^{i\theta_1},\,e^{i\theta_2},\,\ldots,\,e^{i\theta_d}\big).\]
\begin{rem}
	We will sometimes need to define the Bohr lift when the map $\varphi(s)=\sum_{n\geq1} c_n n^{-s}$ is not a Dirichlet polynomial. It is then defined as $$\Phi(z)=\left(c_1-\frac 12\right)+\sum_{n\geq2} c_n z^{\alpha(n)}$$ where we use the substitution $p_j^{-s}\mapsto z_j$. If we assume that $\varphi \in \mathcal{G}$, its Bohr lift $\Phi$ is now well-defined on $\DD^\infty\cap c_0$, and Kronecker's theorem shows that this set is mapped by $\Phi$ into $\mathbb{C}_0$. 
\end{rem}

Let us come back to a polynomial $\varphi\in\mathcal G$. If we assume that $\varphi$ has unrestricted range, there exists at least one point $w\in\mathbb{T}^d$ so that $\mathrm{Re}\,\Phi(w)=0$, by the compactness of $\TT^d$. Let $w=\big(e^{i\vartheta_1},\,e^{i\vartheta_2},\,\ldots,\,e^{i\vartheta_d}\big)$. Then $\vartheta=(\vartheta_1,\,\vartheta_2,\,\ldots,\,\vartheta_d)$ has to be a critical point of $\mathrm{Re}\,\phi$ since this last map admits a minimum at $\vartheta$. Moreover, the mapping properties of $\varphi$ implies that the Hessian matrix of $\mathrm{Re}\,\phi$ at $\vartheta$ should be non-negative. 
\begin{defn}
	We define the \emph{boundary index of} $\Phi$ \emph{at} $w$ as the non-negative integer $J(\Phi,w)$ such that the signature of the Hessian matrix of $\mathrm{Re}\,\phi$ at $\vartheta$ is equal to $\big(J(\Phi,w),0\big)$. 
\end{defn}

With these definitions at hand, we are able to state our main theorem which shows that, when there are mixed terms, the complex dimension does not give enough information and that we need a more careful study of $\varphi$.
\begin{thm}
	\label{thm:main} Let $\varphi(s)=c_1+\sum_{n\geq 2}c_n n^{-s}$ be a Dirichlet polynomial in $\mathcal G$ with unrestricted range. Suppose that its complex dimension $d$ is greater than or equal to $2$, and let $\Phi$ be a minimal Bohr lift of $\varphi$. Assume that 
	\begin{itemize}
		\item either the degree of $\varphi$ is equal to 1 or 2, 
		\item or the degree of $\varphi$ is at least $3$ and for any $w\in\TT^d$, either $\mathrm{Re}\, \Phi(w)>0$ or $\mathrm{Re}\,\Phi(w)=0$ and $J(\Phi,w)\geq 2$. 
	\end{itemize}
	Then $\mathcal{C}_{\varphi}$ is compact on $\mathcal H^2$. Moreover, the result is optimal in the following sense: 
	\begin{itemize}
		\item If the complex dimension of $\varphi$ is equal to $1$, then $\mathcal{C}_{\varphi}$ is never compact. 
		\item There exist polynomials $\varphi\in\mathcal{G}$ of arbitrary complex dimension and of arbitrary degree greater than or equal to 3 such that $\mathcal C_{\varphi}$ is not compact. 
	\end{itemize}
\end{thm}

At this point we should mention that Theorem~\ref{thm:main} does not encompass Theorem~\ref{thm:polynomials}, and we will return to this point later (see Section~\ref{sec:discussion}). However, Theorem~\ref{thm:main} allows us to conclude that for the Dirichlet polynomials $\varphi$ given by \eqref{eq:simplemixedcase}, which have complex dimension and degree equal to $2$, the induced composition operators are compact.

We are also interested in the degree of compactness of our operators, which may be estimated using their approximation numbers. 
\begin{defn}
	Let $H$ be a Hilbert space and let $T\in\mathfrak L(H)$. The $n$th approximation number of $T$, denoted $a_n(T)$, is the distance of $T$ to the operators of rank $<n$. 
\end{defn}
The study of the behaviour of $a_n(\mathcal C_\varphi)$ when $\varphi\in\mathcal G$ is a linear symbol \eqref{eq:linearcase} has been done in \cite{QS14}. In particular, it is shown there that $$\left(\frac{1}{n}\right)^{(d-1)/2}\ll a_n(\mathcal C_\varphi)\ll \left(\frac{\log n}n\right)^{(d-1)/2}$$ where $d$ is the complex dimension of $\varphi$. We will extend this result to a general context. To keep this introduction sufficiently short, we refer to Section \ref{sec:approximationnumber} for our statement, and give only one striking consequence of it: We may distinguish the Schatten classes of linear operators on $\mathcal H^2$ using composition operators induced by polynomial symbols. By definition, a compact linear operator $T$ belongs to the Schatten class $S_p$, for $0<p<\infty$, if $$\|T\|_p^p:=\textrm{Tr}\left(|T|^p\right)=\sum_{n=1}^{\infty}a_n(T)^p <\infty.$$ 
\begin{cor}
	\label{cor:schatten} Let $0<p<q$. There exists a Dirichlet polynomial $\varphi\in\mathcal G$ such that $\mathcal C_\varphi\in S_q\backslash S_p$. 
\end{cor}

Let us end this introduction by mentioning that the the composition operators induced by the maps $\varphi_1$ and $\varphi_2$ have different degrees of compactness. Indeed, we will show that $$\left(\frac 1n\right)^{1/2}\ll a_n\left(\mathcal C_{\varphi_1}\right)\ll \left(\frac{\log n}{n}\right)^{1/2}\qquad \textrm{ and } \qquad \left(\frac 1n\right)^{1/3}\ll a_n\left(\mathcal C_{\varphi_2}\right)\ll \left(\frac{\log n}{n}\right)^{1/3}.$$

\subsection*{Organization} The remainder of this paper is divided into seven sections. 
\begin{itemize}
	\item Section~\ref{sec:thmc0} contains the proof of Theorem~\ref{thm:c0}. The content of this section is independent from that of the following sections. 
	\item In Section~\ref{sec:prelim} we introduce some necessary tools and results needed for the proof of Theorem~\ref{thm:polynomials} and Theorem~\ref{thm:main}. 
	\item Section~\ref{sec:polyproof} is devoted to the proof of Theorem~\ref{thm:polynomials}. 
	\item Section~\ref{sec:mainproof} contains the proof of Theorem~\ref{thm:main}. 
	\item In Section~\ref{sec:lemmaproof} we prove Lemma~\ref{lem:keylemma}, which is the most technical part of Theorem~\ref{thm:main}. 
	\item In Section~\ref{sec:discussion} we discuss the case $\mathrm{deg}(\varphi)\geq3$ and $J(\Phi,w)=0$, its connection to Theorem~\ref{thm:polynomials} and some related examples. 
	\item Finally, in Section~\ref{sec:approximationnumber}, we discuss the decay of the sequence of approximation numbers for some of our operators. 
\end{itemize}

\subsection*{Notation} The notation $f(\veps)\ll g(\veps)$ will mean that $f(\veps)\leq Cg(\veps)$ for some constant $C$ which does not depend on $\veps$. We will sometimes write $f(\veps)\ll_a g(\veps)$ to emphasize that $C$ depends on $a$. As usual, we let $\{p_j\}$ denote the sequence of prime numbers written in increasing order. We let $\mathbf{m}_d$ denote the normalized Lebesgue measure on $\mathbb{T}^d$. This measure is invariant under rotations. If we do not have a priori knowledge of the complex dimension $d$, we will often call this measure $\mathbf{m}_\infty$. For a point $z=e^{i\theta}$ on the unit circle $\TT$, we will always assume that $\theta\in(-\pi,\pi]$. Finally, $\mathbf 0$ will denote the point $(0,\,\ldots,\,0) \in \mathbb{C}^d$, and $\mathbf 1$ will similarly denote the point $(1,\,\ldots,\,1)$.

\section{Proof of Theorem~\ref{thm:c0}} \label{sec:thmc0} Let $\varphi(s)=c_0 s+c_1+\sum_{n=2}^N c_n n^{-s}\in\mathcal G$ such that $c_0\geq 1$. We already know that if $\varphi$ has restricted range, then $\mathcal{C}_\varphi$ is compact. Let us therefore assume that $\mathcal{C}_\varphi$ is compact and also assume that $\varphi$ has unrestricted range, to argue by contradiction. 

By \cite[Thm.~3]{Ba03}, we know that 
\begin{equation}
	\label{EQCOCOMPACTNESS} \frac{\mathrm{Re}\,\varphi(s)}{\mathrm{Re}(s)}\xrightarrow{\mathrm{Re}(s)\to 0}+\infty. 
\end{equation}

Now, since $\varphi$ has unrestricted range there exists a sequence $\{s_k=\sigma_k+it_k\}_{k\geq1}$ in $\mathbb{C}_0$ such that $\mathrm{Re}\,\varphi(s_k) \to 0$. It is well-known that this forces that $\sigma_k \to 0$ (see \cite{Ba03}). Then
\[\mathrm{Re}\,\varphi(s_k)=c_0\sigma_k + \mathrm{Re}(c_1)+\sum_{n=2}^N n^{-\sigma_k}\big(\mathrm{Re}(c_n)\cos(t_k\log(n))+\mathrm{Im}(c_n)\sin(t_k\log(n))\big).\]
By successive extraction of subsequences, we may assume that there exist real numbers $a_n$ and $b_n$ so that for $2 \leq n \leq N$ we have, as $k \to \infty$,
\[\cos(t_k\log(n))\to a_n\qquad \text{ and }\qquad \sin(t_k\log(n))\to b_n.\]
Hence, we may write
\[\mathrm{Re}\,\varphi(s_k)=c_0\sigma_k +\mathrm{Re}(c_1)+\sum_{n=2}^N n^{-\sigma_k}\big(\mathrm{Re}(c_n)a_n+\mathrm{Im}(c_n)b_n\big)+\sum_{n=2}^N n^{-\sigma_k}F_n(t_k),\]
where each $F_n(t_k)\to 0$ as $k \to \infty$. Since $\mathrm{Re}\,s_k=\sigma_k$ also goes to $0$, we may deduce that
\[\mathrm{Re}(c_1)+\sum_{n=2}^N \big(\mathrm{Re}(c_n)a_n+\mathrm{Im}(c_n)b_n\big)=0,\]
so that we have
\[\mathrm{Re}\,\varphi(s) = c_0\sigma + \sum_{n=2}^N \left(n^{-\sigma}-1\right)\big(\mathrm{Re}(c_n)a_n+\mathrm{Im}(c_n)b_n\big) + \sum_{n=2}^N n^{-\sigma}F_n(t).\]
We will now choose another sequence $\left\{s_k' = \sigma_k' + it_k\right\}_{k\geq1}$ where $\mathrm{Re}(s_k') \to 0$ in order to obtain a contradiction with \eqref{EQCOCOMPACTNESS}. More precisely, let $\left\{\sigma'_k\right\}_{k\geq1}$ be any sequence of positive real numbers tending to $0$ such that, for any $n=2,\ldots,N$ and every $k\geq 1$, we have $n^{-\sigma'_k}|F_n(t_k)|\leq\sigma'_k$. Then we obtain
\[\mathrm{Re}\, \varphi(s'_k) = c_0 \sigma_k' + \sum_{n=2}^N \left(n^{-\sigma_k'}-1\right)\big(\mathrm{Re}(c_n)a_n+\mathrm{Im}(c_n)b_n\big) + \sum_{n=2}^N n^{-\sigma'_k}F_n(t_k) = \mathcal{O}(\sigma_k') = \mathcal{O}\big(\mathrm{Re}(s_k')\big),\]
and this contradicts \eqref{EQCOCOMPACTNESS}. The assumption that $\varphi$ has unrestricted range must be wrong. \qed
\begin{rem}
	An inspection of the proof reveals that the statement of Theorem \ref{thm:c0} remains true if we assume that $\varphi(s)=c_0s+c_1+\sum_{n=2}^{\infty}c_n n^{-s}\in\mathcal G$ with $c_0\geq 1$, $\sum_{n=1}^\infty |c_n|<+\infty$ and that the complex dimension of $\varphi$ is finite. The latter assumption is needed to use \eqref{EQCOCOMPACTNESS}. 
\end{rem}

\section{Preliminaries} \label{sec:prelim} As explained in the introduction, our main tool for proving or disproving compactness is a result from \cite{QS14}. We formulate it in a more general context than for polynomials since it will be used under this form in Section \ref{sec:approximationnumber}. Recall that a \emph{Carleson square} in $\CC_0$ is a closed square in $\overline{\CC_0}$ with one of its sides lying on the vertical line $i\RR$; the side length of $Q$ is denoted by $\ell(Q)$. A non-negative Borel measure $\mu$ on $\overline{\CC_0}$ is called a \emph{vanishing Carleson measure} if $$\limsup_{\ell(Q)\to 0}\frac{\mu(Q)}{\ell(Q)}=0.$$
\begin{lem}
	\label{lem:QS} Suppose that $\varphi(s)=\sum_{n\geq 1}c_nn^{-s}\in\mathcal G$ and that $\varphi(\CC_0)$ is bounded. The corresponding composition operator $\mathcal {C}_\varphi$ is compact on $\mathcal H^2$ if and only if the measure
	\[\mu_\varphi(E) := \mathbf{m}_\infty \left(\left\{z \in \mathbb{T}^\infty \,:\, \Phi(z) \in E\right\}\right), \qquad E \subseteq \mathbb{C}_{0}.\]
	is vanishing Carleson in $\mathbb{C}_{0}$, where $\Phi$ denotes a Bohr lift of $\varphi$. 
	\begin{proof}
		This is Corollary 4.1 in \cite{QS14}. 
	\end{proof}
\end{lem}

Hence we consider squares
\[Q=Q(\tau,\varepsilon) = [0,\varepsilon]\times[\tau-\varepsilon/2,\tau+\varepsilon/2],\]
and want to investigate whether $\mu_\varphi(Q) = o(\varepsilon)$ uniformly in $\tau \in \mathbb{R}$. Our next lemma points out that this depends only on the local behaviour of $\Phi$.
\begin{lem}
	\label{lem:compactTd} Let $\varphi$ be a Dirichlet polynomial \eqref{eq:diripoly} with $c_0=0$ mapping $\mathbb{C}_0$ into $\mathbb{C}_{1/2}$ and let $\Phi$ be a minimal Bohr lift of $\varphi$. If for every $w \in \mathbb{T}^d$ with $\mathrm{Re}\, \Phi(w)=0$ there exists a neighbourhood $\mathcal{U}_w \ni w$ in $\mathbb{T}^d$, constants $C_w > 0$ and $\kappa_w > 1$ such that for every $\tau \in \mathbb{R}$ and every $\varepsilon>0$ we have 
	\begin{equation}
		\label{eq:Minf} \mathbf{m}_d \left(\left\{z \in \mathcal{U}_w \, : \, \Phi(z) \in Q(\tau,\,\varepsilon) \right\} \right)\leq C_w \varepsilon^{\kappa_w}, 
	\end{equation}
	then $\mathcal C_{\varphi}$ is compact. 
	\begin{proof}
		Since $\varphi$ is a Dirichlet polynomial, it has finite complex dimension $d$.
		
		We first observe that \eqref{eq:Minf} is always satisfied for those $w\in\TT^d$ with $\mathrm{Re}\,\Phi(w)>0$. Indeed, by continuity of $\Phi$, we may always find a neighbourhood $\mathcal U_w\ni w$ and $\veps_0>0$ such that, for all $\veps\in(0,\veps_0)$ and all $\tau\in\RR$, $\big\{z\in\mathcal U_w\,:\, \Phi(z)\in Q(\tau,\veps)\big\}$ is empty. We may then take $\kappa_w>1$ be arbitrary and choose $C_w$ with $C_w \veps_0^{\kappa_w}\geq 1$.
		
		We will then use a compactness argument and Lemma~\ref{lem:QS}. Indeed, there exists a finite number of points $w_1,\ldots,w_N$ such that $\TT^d$ is covered by $\mathcal U_{w_1},\ldots,\mathcal U_{w_N}$. Now, we may take $C=C_{w_1}+\cdots+C_{w_N}$ and $\kappa=\min(\kappa_{w_1},\ldots,\kappa_{w_N})$. Hence, for all $\tau\in\RR$ and all $\veps>0$, $$\mathbf{m}_d\left(\left\{z\in \mathcal \TT^d \,:\, \Phi(z)\in Q(\tau,\veps)\right\}\right)\leq C\veps^{\kappa}$$ which achieves the proof of the compactness of $\mathcal C_\varphi$ on $\mathcal H^2$. 
	\end{proof}
\end{lem}

Hence, we will require more information about the Taylor coefficients of $\Phi$ at a boundary point. Assume that $\Phi(w)=0$ where $w=\mathbf 1$. In this case, we will rewrite 
\begin{equation}
	\label{eq:newform} \Phi(z) = \sum_{n \in \Lambda} \widetilde{c_n} \prod_{j=1}^d (1-z_j)^{\alpha_j} = \sum_{\alpha\in\NN^d} c_\alpha (1-z)^\alpha, 
\end{equation}
where we have adopted the convention $c_\alpha = \widetilde{c_n}$, which is not generally equal to $c_n$. We shall need a kind of Julia--Caratheodory theorem for $\Phi$ of the form \eqref{eq:newform}.
\begin{lem}
	\label{lem:nonneglin} Let $\Phi:\DD^d\to\CC_0$ be of the form \eqref{eq:newform} and let $|\alpha|=1$. Then $c_\alpha \geq 0$. Moreover, there exists at least one multi-index $\alpha$ with $|\alpha|=1$ and $c_\alpha>0$, unless $\Phi \equiv 0$. 
	\begin{proof}
		We may assume that $\alpha = (1,0,\ldots,0)$. Consider the one-variable polynomial
		\[\psi(w) = \Phi(w,1,\ldots,1).\]
		Clearly, $\psi$ maps $\mathbb{D}$ to $\mathbb{C}_0$, and $\psi(1)=0$. We write
		\[\psi(w) = a(1-w)+b(1-w)^2 + \mathcal{O}\left((1-w)^3\right).\]
		We set $w = e^{i\theta}$ and obtain
		\[\psi\big(e^{i\theta}\big) = a \left(\frac{\theta^2}{2}-i\theta\right)-b\theta^2 + \mathcal{O}\left(\theta^3\right).\]
		In particular,
		\[\mathrm{Re}\,\big(\psi\big(e^{i\theta}\big)\big) = \theta\mathrm{Im}\,(a) + \theta^2\left(\frac{\mathrm{Re}\,(a)}{2}-\mathrm{Re}\,(b)\right) + \mathcal{O}\left(\theta^3\right).\]
		Since this should be non-negative, clearly $\mathrm{Im}(a)=0$. We now set $w = 1-\delta$ for $0 < \delta < 1$ and consider $\psi(\delta) = a\delta + \mathcal{O}\left(\delta^2\right)$. Since the real part of this also should be non-negative as $\delta \to 0^+$ we must have $a\geq0$. Hence $c_\alpha\geq 0$ when $|\alpha|=1$. 
		
		Now, consider the mapping
		\[\alpha \mapsto n(\alpha) = \prod_{j=1}^d p_j^{\alpha_j}.\]
		It defines a total order on $\NN^d$ by setting $\alpha \leq \beta$ if and only if $n(\alpha)\leq n(\beta)$. Assume that $\Phi\not\equiv 0$ and that $c_\alpha=0$ whenever $|\alpha|=1$. Consider
		\[\beta = \inf\left\{\alpha \,:\, c_\alpha\neq0\right\},\]
		which exists since $\Phi \not\equiv 0$. There is $\theta \in (-\pi,\pi]$ so that $c_\beta = |c_\beta| e^{i\theta}$. Fix $\theta_j \in (-\pi/2,\pi/2)$ and define
		\[z_j = 1-p_j^{-\sigma}e^{i\theta_j},\]
		where $\sigma>0$. For large enough $\sigma$, clearly $z=(z_1,\,\ldots,\,z_d)\in\mathbb{D}^d$. Moreover,
		\[\Phi(z_1,\,\ldots,\,z_d) = |c_\beta|e^{i\theta}[n(\beta)]^{-\sigma} e^{i(\beta_1\theta_1+\cdots+\beta_d\theta_d)} + o\left([n(\beta)]^{-\sigma}\right),\]
		as $\sigma \to \infty$. This implies that
		\[\mathrm{Re}\, \left(\Phi(z_1,\,\ldots,\,z_d)\right) = |c_\beta|[n(\beta)]^{-\sigma}\cos(\theta + \beta_1\theta_1 + \cdots \beta_d \theta_d) + o\left([n(\beta)]^{-\sigma}\right).\]
		Since $|\beta|\geq2$, we can always choose $\theta_j \in (-\pi/2,\pi/2)$ such that $\cos(\theta + \beta_1\theta_1 + \cdots + \beta_d \theta_d)<0$. This contradicts the mapping properties of $\Phi$, and hence the assumption that $c_\alpha=0$ whenever $|\alpha|=1$ is wrong. 
	\end{proof}
\end{lem}

We will also need two lemmas from differential geometry. The first one is the parametrized Morse lemma (see for instance \cite[Sec.~4.44]{BG92}). 
\begin{lemun}
	[Parametrized Morse Lemma] Let $\mathcal U\subset\mathbb R^J\times\RR^{d-J}$ be a neighbourhood of $\mathbf{0}\in\mathbb R^d$ and let $F:\mathcal U\to\mathbb R,\ (u,v)\mapsto F(u,v)$ be a smooth function. Assume that $F(\mathbf 0)=0$, that $
	\partial F/
	\partial u_i(\mathbf 0)=0$ for all $i=1,\ldots,J$ and that the matrix
	\[\left(\frac{
	\partial^2 F}{
	\partial u_i
	\partial u_j}(\mathbf 0)\right)_{1\leq i,j\leq J}\]
	is positive definite. Then there exist a neighbourhood $\mathcal V\ni\mathbf{0}$ with $\mathcal V\subset\mathcal U$, a smooth diffeomorphism $\Gamma:\mathcal V\to \RR^d,\ (u,v)\mapsto (\gamma(u,v),v)$ with $\Gamma(\mathbf{0})=\mathbf{0}$ and a smooth map $h:\RR^{d-J}\to\mathbb R$ such that, for any $(u,v)\in\mathcal V$,
	\[F(u,v)=\sum_{j=1}^J \gamma_j(u,v)^2+h(v).\]
\end{lemun}

The second lemma reads as follows. 
\begin{lem}
	\label{LEMVOLUME1} Let $p\geq 1$ be an integer, and let $f:I\to\mathbb R$ be a smooth function where $I$ is an open interval containing $0$ and $f(x)\sim_0 x^p$. Then there exist $C>0$ and an open interval $I'\ni 0$ inside $I$ such that, for any $\tau\in\mathbb R$ and any $\delta>0$, the set $\left\{x\in I'\,:\,|f(x)-\tau|<\delta\right\}$ has Lebesgue measure less than $C\delta^{1/p}$. 
	\begin{proof}
		Assume first that $f(x)=x^p$. If $|\tau|\leq 2\delta$, then the result is clear. Otherwise, if $\tau\geq2\delta$, then $x$ has to live in $\big[(\tau-\delta)^{1/p},(\tau+\delta)^{1/p}\big]$ and the length of this interval may be easily estimated using the mean value theorem.
		
		The general case reduces to this one. For small values of $x$, set $y=[f(x)]^{1/p}$ if $p$ is odd or $y=[f(x)]^{1/p}$ for $x>0$, $y=-[f(x)]^{1/p}$ for $x<0$ if $p$ is even. In both cases, $y$ is differentiable at $0$ and $dy/dx>0$. Hence, $x=\gamma(y)$ where $\gamma$ is a smooth diffeomorphism. Now, for some small open interval $I'\ni 0$, we have
		\[\left\{x\in I'\,:\, |f(x)-\tau|<\delta\right\}=\left\{x\in I'\,:\, |\big(\gamma^{-1}(x)\big)^p-\tau|<\delta\right\}.\]
		Since $\gamma$ is a diffeomorphism, the latter set has Lebesgue measure less than $C\delta^{1/p}$. 
	\end{proof}
\end{lem}

\section{Proof of Theorem~\ref{thm:polynomials}} \label{sec:polyproof} We intend to apply Lemma~\ref{lem:compactTd}. Hence, let $w\in\TT^d$ with $\mathrm{Re}\ \Phi(w)=0$. By the rotational invariance of $\mathbf{m}_d$, we may always assume that $w=\mathbf 1$. Moreover, since the conditions in Lemma~\ref{lem:compactTd} are invariant by vertical translations, we may also assume that $\Phi(w)=0$. In this case we have
\[\Phi(z_1,\,z_2,\,\ldots,\,z_d) = \sum_{j=1}^d \Phi_j(z_j) = \sum_{j=1}^d \sum_{k} a_k^{(j)} (1-z_j)^k.\]
Since $\Phi$ is a minimal Bohr lift of $\varphi$, inspecting the proof of Lemma~\ref{lem:nonneglin}, we may conclude that in this case $a_1^{(j)}>0$ for every $j=1,\,2,\,\ldots,\,d$. This means we have 
\begin{align*}
	\mathrm{Re}\, \Phi\big(e^{i\theta_1},\,e^{i\theta_2},\,\ldots,\,e^{i\theta_d}\big) &= \sum_{j=1}^d b_j \theta_j^{k_j} + o\left(\theta_j^{k_j}\right), \intertext{where the coefficients $b_j\neq0$ are real numbers and the exponents $k_j\geq2$ are integers. The fact that this quantity is supposed to be non-negative implies that $b_j>0$ and that $k_j$ is even, by similar considerations as those in the proof of Lemma~\ref{lem:nonneglin}. Moreover} \mathrm{Im}\, \Phi\big(e^{i\theta_1},\,e^{i\theta_2},\,\ldots,\,e^{i\theta_d}\big) &= -\sum_{j=1}^d a_1^{(j)} \theta_j + o\big(\theta_j\big). 
\end{align*}
\begin{proof}
	[Proof of the first part of Theorem~\ref{thm:polynomials}] Let $\tau \in \mathbb{R}$ and $\varepsilon>0$ be arbitrary. The preceding discussion means there is some neighbourhood $\mathcal{U}\ni (1,1,\ldots,1)$ in $\mathbb{T}^d$ so that
	\[\frac{1}{2} \sum_{j=1}^d b_j \theta_j^{k_j} \leq \mathrm{Re}\, \Phi\big(e^{i\theta_1},\,e^{i\theta_2},\,\ldots,\,e^{i\theta_d}\big) \leq 2 \sum_{j=1}^d b_j \theta_j^{k_j},\]
	when $e^{i\theta} \in \mathcal{U}$. Hence if $\Phi\big(e^{i\theta}\big) \in Q(\tau,\varepsilon)$ and $e^{i\theta} \in \mathcal{U}$, we conclude from the real part that $|\theta_j| \ll \varepsilon^{1/k_j},$ for $j = 1,\,2,\,\ldots,\,d$. Now, fixing $\theta_j$ for $j = 2,\,\ldots,\,d$ we conclude from the imaginary part and Lemma \ref{LEMVOLUME1} that $\theta_1$ can live in an interval of size at most $C\varepsilon$. Hence we have
	\[\mathbf{m}_d \left(\left\{z \in \mathcal{U}_w \,: \, \Phi(z) \in Q(\tau,\varepsilon)\right\}\right) \ll_w \varepsilon^{1 + 1/k_2 + \cdots + 1/k_d}.\]
	In fact, we may choose
	\[\kappa_w = 1 + \sum_{j=1}^d \frac{1}{k_j} - \min_{1 \leq j \leq d} \frac{1}{k_j},\]
	and conclude by Lemma~\ref{lem:compactTd}, since $d \geq 2$. 
\end{proof}
\begin{proof}
	[Proof of the second part of Theorem~\ref{thm:polynomials}] In this case $d=1$, and the polynomial $\Phi(z)$ is of only one variable. We again consider some neighbourhood $\mathcal{U}\ni 1$ in $\mathbb{T}$, so that when $e^{i\theta}\in \mathcal{U}$ we have
	\[ 0 \leq \mathrm{Re}\,\Phi\big(e^{i\theta}\big) \leq 2b \theta^k \qquad \text{and} \qquad \big|\mathrm{Im}\,\Phi\big(e^{i\theta}\big)\big| \leq 2a|\theta|,\]
	where $a=a_1$, $b=b_1$ and $k\geq2$ is even. Now, we choose $\tau=0$ and observe that $\varphi\big(e^{i\theta}\big)$ belongs to $Q(\tau,\veps)$ provided $|\theta|\ll\veps$. Hence
	\[\mathbf{m}_1 \left(\left\{z \in \mathbb{T} \,: \, \Phi(z) \in Q(\tau,\varepsilon)\right\}\right) \geq \mathbf{m}_1\left(\left\{z \in \mathcal{U}_w \,: \, \Phi(z) \in Q(\tau,\varepsilon)\right\}\right) \gg \varepsilon,\]
	and $\mathcal{C}_\varphi$ cannot be compact by Lemma~\ref{lem:QS}. 
\end{proof}
\begin{rem}
	Inspecting the proof of Theorem~\ref{thm:polynomials}, we see that we may replace the polynomials $P_j$, by corresponding power series
	\[P_j\big(q_j^{-s}\big) = \sum_{k=0}^\infty c_k^{(j)} q_j^{-ks},\]
	provided $\sum_{k=0}^\infty \big|c_k^{(j)}\big| < \infty$. However, we still require the complex dimension $d$ to be finite. 
\end{rem}

\section{Proof of Theorem~\ref{thm:main}} \label{sec:mainproof} We begin by observing that the penultimate point of Theorem~\ref{thm:main} follows from the second part of Theorem~\ref{thm:polynomials}. Regarding the final part of Theorem~\ref{thm:main}, it is contained in the following result: 
\begin{lem}
	\label{lem:Cex3} There are polynomials $\varphi\in\mathcal G$ of any complex dimension and of any degree $\geq 3$ for which the corresponding composition operator $\mathcal {C}_\varphi$ is non-compact. 
	\begin{proof}
		Let $P(z) = P(z_1,z_2,\ldots,z_d)$ be any polynomial in $d$ variables and define
		\[\Phi(z) = (1-z_1)+\delta(1-z_1)^2P(z),\]
		for some $\delta>0$ to be decided later. We compute
		\[\mathrm{Re}\,\Phi\big(e^{i\theta_1},\ldots,e^{i\theta_d}\big) = (1-\cos{\theta_1})\left(1-2\delta\left(\cos{(\theta_1)}\mathrm{Re}\,P\big(e^{i\theta}\big)-\sin{(\theta_1)}\mathrm{Im}\,P\big(e^{i\theta}\big)\right)\right).\]
		Pick $\delta$ small enough so that we have
		\[\frac{1-\cos{\theta_1}}{2} \leq \mathrm{Re}\,\Phi\big(e^{i\theta_1},\ldots,e^{i\theta_d}\big) \leq 2(1-\cos{\theta_1}).\] 
The first inequality tells us that $\Phi$ is a minimal Bohr lift of
		\[\varphi(s)=\big(1-{p_1}^{-s}\big)+\delta\big(1-{p_1}^{-s}\big)^2P\big(p_1^{-s},\ldots,p_d^{-s}\big),\]
		with $\varphi \in \mathcal{G}$ having unrestricted range.  Using the second inequality and a Taylor expansion of $\mathrm{Im}\,\Phi$, we also get that near $\mathbf 1$, 
\begin{align*}
		\mathrm{Re}\,\Phi\big(e^{i\theta_1},\ldots,e^{i\theta_d}\big)&=\mathcal O\left(\theta_1^2\right), \\
		\mathrm{Im}\,\Phi\big(e^{i\theta_1},\ldots,e^{i\theta_d}\big)&=\mathcal O(\theta_1). 
	\end{align*}
		Similar considerations as in the proof of the second part of Theorem~\ref{thm:polynomials} allow us to conclude that $\mathcal C_{\varphi}$ is not compact. 
	\end{proof}
\end{lem}
\begin{rem}
	The key point of Lemma~\ref{lem:Cex3} is that even if $\Phi$ involves $d$ variables, its local behaviour near $\mathbf 1$ depends too heavily on $z_1$ to ensure compactness. 
\end{rem}

Having now concluded the negative parts of Theorem~\ref{thm:main}, we turn to the positive parts. Let us fix a polynomial $\varphi \in \mathcal G$ and let $\Phi$ denote a minimal Bohr lift of $\varphi$. We can simplify how to write $\Phi$ around a point $w\in\TT^d$ such that $\mathrm{Re}\,\Phi(w)=0$. Without loss of generality, we may again assume that $w=\mathbf 1$ and that $\Phi(w)=0$. Then we may write
\[\Phi(z) = \sum_{j=1}^d a_j (1-z_j) + \sum_{j=1}^d b_j(1-z_j)^2 + \sum_{1 \leq j < k \leq d} c_{j,k}(1-z_j)(1-z_k)+o\left(\sum_{1\leq j\leq d}|1-z_j|^2\right).\]
We let $z_j = e^{i\theta_j}$ and since $a_j \geq 0$ by Lemma~\ref{lem:nonneglin} we get 
\begin{align*}
	\mathrm{Re}\,\left(\Phi(z)\right) &= \sum_{j=1}^d \left(\frac{a_j}{2}-\mathrm{Re}\,(b_j)\right)\theta_j^2-\sum_{1 \leq j < k \leq d} \mathrm{Re}\,(c_{j,k})\theta_j\theta_k + o\left(\sum_{1\leq j\leq d}\theta_j^2\right). \intertext{The quadratic form appearing above is brought to standard form by a linear change of variables,} \mathrm{Re}\,\left(\Phi(z)\right)&= \sum_{j=1}^d \left(\ell_j(\theta)\right)^2 +o\left(\sum_{1\leq j\leq d}\theta_j^2\right). \intertext{Next, we write} \mathrm{Im}\left(\Phi(z)\right) &= - \sum_{j=1}^{d} a_j \theta_j + o\left(\sum_{j=1}^d |\theta_j|\right) = -\ell_{d+1}(\theta_j) + o\left(\sum_{j=1}^d |\theta_j|\right), 
\end{align*}
and by Lemma~\ref{lem:nonneglin} we know that $\ell_{d+1}\not\equiv0$, since at least one $a_j>0$. The last step to finish the proof of Theorem \ref{thm:main} is the following result:
\begin{lem}
	\label{PROPMAIN} Let $\Phi:\DD^d\to\CC_0$ be an optimal Bohr lift of $\varphi \in \mathcal{G}$, where $\varphi$ has unrestricted range and complex dimension $d\geq2$. Suppose that $w\in\mathbb T^d$ is such that $\mathrm{Re}\,\Phi(w)=0$. Then there exist a neighbourhood $\mathcal U_w\ni w$ in $\mathbb{T}^d$, $\kappa=\kappa_w>1$ and $C=C_w>0$ such that, for any $\tau\in\mathbb R$ and for every $\veps>0$,
	\[\mathbf{m}_d \left(\left\{z\in \mathcal U_w\,:\, \Phi(z)\in Q(\tau,\veps)\right\}\right)\leq C\veps^\kappa\]
	in the following cases: 
	\begin{itemize}
		\item $J(\Phi,w)\geq 1$ and $\ell_{d+1}$ is independent from $(\ell_1,\ldots,\ell_J)$. We may choose $\kappa=1+J(\Phi,w)/2$. 
		\item $J(\Phi,w)\geq 2$ and $\ell_{d+1}$ belongs to $\mathrm{span}(\ell_1,\ldots,\ell_J)$. We may choose $\kappa=(1+J(\Phi,w))/2$. 
		\item $J(\Phi,w)=1$, $\ell_{d+1}$ is a multiple of $\ell_1$ and $\Phi$ has degree $2$. We may choose $\kappa=9/8$. 
		\item $J(\Phi,w)=0$ and $\Phi$ has degree $2$. We may choose $\kappa=(d+3)/4$. 
	\end{itemize}
\end{lem}
Before we prove the different cases of this lemma, let us make some comments. Firstly, it is clear that Lemma~\ref{PROPMAIN} and Lemma~\ref{lem:compactTd} imply Theorem~\ref{thm:main} when the degree of $\varphi$ is at least 2. When the degree of $\varphi$ is equal to 1, then
\[\Phi(z)=\sum_{j=1}^d a_j(1-z_j)\]
so that each $a_j$ is positive. This implies that $J(\Phi,w)=d$ so that we may again apply Lemma~\ref{PROPMAIN} and Lemma~\ref{lem:compactTd}.

It is also important to notice that $\Phi$ cannot be an arbitrary polynomial mapping of $\DD^d$ into $\CC_0$. It is an optimal Bohr lift of some $\varphi\in\mathcal G$ with complex dimension $d$. In particular, we shall use that $\frac{
\partial \Phi}{
\partial z_j}\not\equiv 0$ for every $1\leq j\leq d$. Moreover, the polynomial $\Phi(z)=\lambda(1-z_1z_2)$ is not an optimal Bohr lift. Otherwise, it would arise from $\varphi(s)=\lambda(1-q_1^{-s}q_2^{-s})$, but the optimal Bohr lift of $\varphi$ is $\lambda(1-z)$.

We are now ready for the proof of Lemma~\ref{PROPMAIN}. By similar considerations as before, we may again assume that $w=\mathbf 1$ and that $\Phi(w)=0$. We write $J$ for $J(\Phi,w)$.

\subsection*{The case $J=0$} This implies that
\[\frac{a_j}{2}-\mathrm{Re}\,(b_j) = \mathrm{Re}\,(c_{j,k}) = 0,\]
for $j,k=1,\,\ldots,\,d$. We set $z_j = e^{i\theta_j}$ and compute 
\begin{align*}
	\mathrm{Re}\,\left(a_j(1-z_j)\right) &= a_j(1-\cos{\theta_j}) \\
	\mathrm{Re}\,\left(b_j(1-z_j)^2\right) &= -a_j\cos{\theta_j}(1-\cos{\theta_j})+2\mathrm{Im}(b_j)\sin{\theta_j}(1-{\cos{\theta_j}}) \\
	\mathrm{Re}\,\left(c_{j,k}(1-z_j)(1-z_k)\right) &= \mathrm{Im}\left(c_{j,k}\right)\left(\sin{\theta_j}\left(1-\cos{\theta_k}\right)+\sin{\theta_k}(1-\cos{\theta_j})\right) 
\end{align*}
This means that
\[\mathrm{Re}\,\left(\Phi(z)\right) = \sum_{j=1}^d \mathrm{Im}(b_j)\theta_j^3 + \sum_{1 \leq j < k \leq d} \frac{\mathrm{Im}\left(c_{j,k}\right)}{2}\left(\theta_j\theta_k^2 + \theta_k\theta_j^2\right) + o\left(\sum_{j=1}^d |\theta_j|^3\right).\]
However, the non-negativity of $\mathrm{Re}\,\Phi$ then implies that $\mathrm{Im}(b_j)=\mathrm{Im}\left(c_{j,k}\right)=0$. Hence we in total have $b_j = a_j/2$ and $c_{j,k}=0$, which means
\[\Phi(z) = \sum_{j=1}^d \left(a_j(1-z_j) + \frac{a_j}{2}(1-z_j)^2\right).\]
In fact, this means that $a_j>0$ for every $j$, by the assumption that the complex dimension is $d$ and Lemma~\ref{lem:nonneglin}. We may now use (the proof of) Theorem~\ref{thm:polynomials} to conclude that there exists a neighbourhood $\mathcal U_w\ni w$ such that
\[\mathbf{m}_d\left(\left\{z\in\mathcal U_w\,:\, \Phi(z)\in Q(\tau,\varepsilon)\right\}\right) \ll \varepsilon \times \varepsilon^\frac{d-1}{4} = \varepsilon^\frac{d+3}{4},\]
since we now have
\[\mathrm{Re}\, \Phi\big(e^{i\theta_1},\ldots,e^{i\theta_d}\big)=\frac{1}{4}\sum_{j=1}^d a_j \theta_j^4+o\left(\theta_j^4\right),\]
and we are done with this case. \qed

\subsection*{The case $J\geq1$ and independence} After a linear change of variables, we may write $\mathrm{Re}\,\phi$ and $\mathrm{Im}\,{\phi}$ as 
\begin{eqnarray*}
	\mathrm{Re}\,\phi(\theta_1,\ldots,\theta_d)&=&u_1^2+\cdots+u_J^2+o\left(\sum_{j=1}^du_j^2\right)\\
	\mathrm{Im}\,\phi(\theta_1,\ldots,\theta_d)&=&u_{d}+o\left(\sum_{j=1}^d |u_j|\right) 
\end{eqnarray*}
Since a linear change of variables does not change the value of the volume up to constants, we may assume that $\phi$ depends on $(u,v)$ with $u=(u_1,u_2,\ldots,u_J)$ and $v=(u_{J+1},\ldots,u_d)$. Applying the parametrized Morse lemma to $\mathrm{Re}\,\phi$, we may write 
\begin{align*}
	\mathrm{Re}\, \phi(u,v)&=\gamma_1(u,v)^2+\cdots+\gamma_J(u,v)^2+h(v). \intertext{We also apply the change of variables $(u,v)\mapsto \Gamma(u,v)$ to $\mathrm{Im}\,\phi$ and since $\Gamma_d(u,v)=u_d$, we find} \mathrm{Im}\, \phi(u,v)&=u_{d}+g(\Gamma(u,v)), 
\end{align*}
where $g$ is a smooth function defined on $\mathcal{V}$ such that $
\partial g/
\partial u_{d}(\mathbf 0)=0.$

Now, we know that $\mathrm{Re}\,\phi(u,v)\geq 0$ for every $(u,v)\in \RR^d$. Since $\Gamma$ is a diffeomorphism, $v$ can take any value in some neighbourhood of zero in $\RR^{d-J}$ even if we require that
\[\gamma_1(u,v)=\gamma_2(u,v) =\cdots=\gamma_J(u,v)=0,\]
and hence $h(v)\geq 0$. 

This implies that we may find some neighbourhood $\mathcal W \ni \mathbf{0}$ in $\mathcal V$ such that, for every $\tau\in\mathbb R$ and every $\veps>0$,
\[(u,v)\in\mathcal W\,\quad \text{and}\quad \phi(u,v)\in Q(\tau,\veps) \qquad \implies\qquad |\gamma_j(u,v)|\leq \veps^{1/2}.\]
Now, for if we fix $\gamma_1(u,v),\,\ldots\,,\gamma_{d-1}(u,v)$, it follows from Lemma~\ref{LEMVOLUME1} with $p=1$ that $\gamma_d(u,v)=u_d$ has to belong to some interval of size $C\veps$, provided that $(u,v)$ is sufficiently close to $\mathbf 0$. This means that there exists a neighbourhood $\mathcal O\subset\mathcal W$ of $\mathbf 0$ such that
\[\left\{(u,v)\in\mathcal O \,:\, \phi(u,v)\in Q(\tau,\veps)\right\}\subset\big\{(u,v)\in\mathcal O\,:\, \Gamma(u,v)\in R(\tau,\veps)\big\},\]
where the volume of $R(\tau,\veps)$ is less than $C\veps^{1+\frac J2}$. Since $\Gamma$ is a diffeomorphism, we are done. \qed

\subsection*{The case $J\geq2$ and dependence} With a similar linear change of variables as in the previous case, we may write 
\begin{eqnarray*}
	\mathrm{Re}\,\phi(u_1,\,\ldots\,,u_d)&=&u_1^2+\cdots+u_J^2+o\left(\sum_{j=1}^du_j^2\right)\\
	\mathrm{Im}\,\phi(u_1,\,\ldots\,,u_d)&=&\sum_{j=1}^J \alpha_j u_{j}+o\left(\sum_{j=1}^d |u_j|\right) 
\end{eqnarray*}
We use again the parametrized Morse lemma with $\mathrm{Re}\, \phi$, and it is again easy to show that $\gamma_j(u,v)=u_j+o\left(\sum_{j=1}^d |u_j|\right)$ so that $$\mathrm{Im}\, \phi(u,v)=\sum_{j=1}^J \alpha_j \gamma_j(u,v)+g(\Gamma(u,v))$$ with ${
\partial g}/{
\partial u_j}(\mathbf 0)=0$ for $j=1,\ldots,d$.

We argue as in the previous case. For every $j=2,\ldots,J$, for any $\tau\in\mathbb R$ and every $\veps>0$,
\[(u,v)\in\mathcal W\subset\mathcal V\quad \text{and} \quad \phi(u,v)\in Q(\tau,\veps)\qquad \implies \qquad |\gamma_j(u,v)|\leq \veps^{1/2}.\]
Now, for a fixed value of $\gamma_2(u,v),\ldots,\gamma_{d}(u,v)$, it is again clear that $\gamma_1(u,v)$ has to belong to some interval of size $C\veps$, provided $(u,v)$ is sufficiently close to $\mathbf 0$. This means that there exists a neighbourhood $\mathcal O \ni \mathbf{0}$ in $\mathcal W$ such that
\[\left\{(u,v)\in\mathcal O\,:\, \phi(u,v)\in Q(\tau,\veps)\right\}\subset\big\{(u,v)\in\mathcal O\,:\,\Gamma(u,v)\in R(\tau,\veps)\big\},\]
where the volume of $R(\tau,\veps)$ is less than $C\veps^{1+\frac {J-1}2}$. We conclude as in the previous step. \qed

\subsection*{The case $J=1$ and dependence, $d=2$} This is the most difficult case. At first, we do not assume that $d=2$ but we always assume that the degree of $\varphi$ is equal to 2. We know that there is constant $\lambda\in\mathbb R^*$ so that $\ell_1(\theta) = \lambda \ell_{d+1}(\theta)$, which means
\[\sqrt{\frac{a_j}{2}-\mathrm{Re}\,(b_j)}=\lambda a_j, \qquad 1\leq j\leq d\]
and that $\lambda>0$ by the computations in the beginning of this section. We normalize $\Phi(z)$ as $\lambda^{-2}\Phi(z)$, so that we may assume that $\lambda=1$. Hence
\[\ell_1(\theta) = \sum_{j=1}^d a_j\theta_j,\]
and this immediately implies that 
\begin{equation}
	\label{eq:relationsJ1} \mathrm{Re}\,(b_j) = \frac{a_j}{2}-a_j^2 \qquad \text{ and } \qquad \mathrm{Re}\,(c_{j,k}) = -2a_j a_k,\qquad 1\leq j,k\leq d. 
\end{equation}
Suppose that $a_1=0$. Then $\mathrm{Re}\,(b_1)=0$ and $\mathrm{Re}\,(c_{1,k})=0$ for $2 \leq k \leq d$. We compute
\[\mathrm{Re}\,\big(\Phi(e^{ix},1,\ldots,1)\big) = -2\mathrm{Im}(b_1)\sin{x}(1-\cos{x})\geq0,\]
which means that $\mathrm{Im}(b_1)=0$, so that $b_1=0$. Next we compute 
\begin{align*}
	\Phi\big(e^{ix},e^{iy},1\ldots,1\big) &= a_2(1-\cos{y}) + \left(\frac{a_2}{2}-a_2^2\right)\left(1-2\cos{y}+\cos{2y}\right) \\
	&\qquad\qquad\qquad\qquad- \mathrm{Im}(c_{1,2})\big(-\sin{x}-\sin{y}+\sin(x+y)\big) \\
	&= (1-\cos{y})\left(a_2(1-\cos{y})+2a_2^2\cos{y}+ \mathrm{Im}(c_{1,2})\sin{x}\right) \\
	&\qquad\qquad\qquad\qquad+ \mathrm{Im}(c_{1,2})\sin{y}(1-\cos{x}). 
\end{align*}
Taking $y = \pm \delta$ for small enough $\delta$, we obtain that $\mathrm{Im}(c_{1,2})=0$. There is nothing special about $z_2$, and hence we conclude that $\mathrm{Im}(c_{1,k})=0$, for $2 \leq k \leq d$. In particular, $c_{1,k}=0$ for the same values of $k$. But this is impossible, since the variable $z_1$ no longer appear in our polynomial. Hence the assumption that $a_1=0$ must be wrong. 

Arguing in the same way, we have that $a_j > 0$ for $1 \leq j \leq d$. Moreover, after renaming the variables, we may suppose $a_1 \geq a_2 \geq \cdots a_d >0$. Finally,
\[0 \leq \mathrm{Re}\,\left(\Phi(-1,1,\ldots,1)\right)= 2a_1 + 4\left(\frac{a_1}{2}-a_1^2\right)\qquad \implies\qquad a_1 \leq 1,\]
so without loss of generality, we may assume that
\[1\geq a_1 \geq a_2 \geq \cdots \geq a_d > 0.\]

From now on, we assume that $d=2$ and that $1\geq a_1\geq a_2>0$. We need the following lemma. 
\begin{lem}
	\label{lem:a2a1} We have $a_2\leq 1-a_1$. 
	\begin{proof}
		We compute $$\Phi(-1,-1)=-4a_1^2-4a_2^2-8a_1a_2+4a_1+4a_2=4(a_1+a_2)(1-a_1-a_2).$$ Since this has to be non-negative, we get the result. 
	\end{proof}
\end{lem}
\begin{rem}
	Lemma~\ref{lem:a2a1} immediately implies that $a_1 \in (0,1)$ and $a_2 \in (0,1/2]$ by the assumptions that $0 < a_2 \leq a_1 \leq 1$. 
\end{rem}

Let us apply the change of variables $\theta_1=a_2u+a_2v$, $\theta_2=a_1u-a_1v$ to $\phi$: 
\begin{eqnarray}
	\mathrm{Re}\,\phi(u,v)&=&-4a_1^2a_2^2 u^2+o(u^2+v^2) \label{EQJ1-3}\\
	\mathrm{Im}\,\phi(u,v)&=&2a_1a_2u+o(|u|+|v|) \label{EQJ1-4}. 
\end{eqnarray}
As before, we intend to apply the parametrized Morse lemma to $\mathrm{Re}\,\phi$. Setting $\Psi=\Gamma^{-1}$, we get that, around $\mathbf 0$, 
\begin{eqnarray*}
	\mathrm{Re}\, \phi\circ\Psi(u,v)&=&u^2+h(v)\\
	\mathrm{Im}\, \phi\circ\Psi(u,v)&=&u+g(u,v) 
\end{eqnarray*}
with $h$ and $g$ smooth functions which have no terms of order 1 at $\mathbf 0$. 

Assume first that $h\not \equiv 0$. Let $p\geq 2$ be such that $h(v)\sim_0 \alpha_p v^p$ with $\alpha_p\neq 0$. Because $\phi\circ\Psi$ maps $\mathbb{R}^2$ into $\overline{\CC_0}$, we must have that $\alpha_p>0$ and that $p$ is even. Now, if $\phi\circ\Psi(u,v)\in Q(\tau,\veps)$ with $(u,v)$ sufficiently close to $\mathbf 0$, then $0\leq h(v)\leq \veps$ which implies by Lemma~\ref{LEMVOLUME1} that $v$ belongs to some set of measure less than $C\veps^{1/p}$. Moreover, for a fixed value of $v$, a look at the imaginary part and Lemma~\ref{LEMVOLUME1} yield that $u$ has to belong to some interval of size $C \veps$ and thus we are done with $\kappa=1+1/p$.

Thus, we are lead to study what happens if $h\equiv0$. The situation is easier if the Taylor expansion of $g(u,v)$ admits some term in $v^p$. In that case, we may write
\[\mathrm{Im}\, \phi\circ\Psi(u,v)=ug_1(u,v)+v^p g_2(v)\]
with smooth functions $g_1$ and $g_2$, such that $g_1(0,0)=1$ and $g_2(0)\neq 0$. If $\phi\circ\Psi(u,v)$ belongs to $Q(\tau,\veps)$, we conclude from the real part that then $|u|\leq\veps^{1/2}$, and from the imaginary part, we get that, near $\mathbf 0$,
\[|v^p g_2(v)-\tau|\leq C\veps^{1/2}.\]
By appealing again to Lemma~\ref{LEMVOLUME1}, we conclude that $v$ belongs to some set of Lebesgue measure less than $C\veps^{1/2p}$. For a fixed value of $v$, we look once more at the imaginary part, and obtain that $u$ must belongs to some interval of size $C\veps$. Hence, we are done with $\kappa=1+1/(2p)$.

Therefore, it remains to show that we will always fall into one of the previous cases and compute the value of $p$. We again recall that the polynomial
\[\Phi(z)=\lambda(1-z_1z_2),\]
is a contradiction to the fact that $\Phi$ is a minimal Bohr lift of $\varphi \in \mathcal{G}$. More precisely, we are reduced to proving the following lemma.
\begin{lem}
	\label{lem:keylemma} Let $0 < a_2 \leq a_1\leq1$ and $a_2 \leq 1- a_1$. Suppose that $\Phi:\DD^2\to\CC_0$ is the polynomial 
	\begin{equation}
		\Phi(z)=a_1(1-z_1)+a_2(1-z_2)+b_1(1-z_1)^2+b_2(1-z_2)^2+c(1-z_1)(1-z_2), \label{EQJ1-5} 
	\end{equation}
	where
	\[\mathrm{Re}(b_1)=\frac{a_1}2-a_1^2,\qquad \mathrm{Re}(b_2)=\frac{a_2}2-a_2^2 \qquad \text{and} \qquad \mathrm{Re}(c)=-2a_1a_2.\]
	Set $\theta_1=a_2u+a_2v$, $\theta_2=a_1u-a_1v$ and
	\[\phi(u,v)=\Phi\big(e^{i\theta_1},e^{i\theta_2}\big).\]
	Then there does not exist smooth maps $\gamma:\RR^2\to\RR$ and $h:\RR^2\to\RR$ so that 
	\begin{eqnarray}
		\mathrm{Re}\,\phi(u,v)&=&\gamma(u,v)^2 \label{EQJ1-1}\\
		\mathrm{Im}\,\phi(u,v)&=&\gamma(u,v)h(u,v)\label{EQJ1-2} 
	\end{eqnarray}
	except if $\Phi(z)=\frac{1}{2}(1-z_1z_2).$ More precisely, if $\Phi(z)\neq \frac 12(1-z_1z_2)$, for any smooth maps $\gamma:\RR^2\to \RR$ and $h:\RR\to\RR$, then 
	\begin{itemize}
		\item either the Taylor series of $\mathrm{Re}\,\phi-\gamma^2$ at $\mathbf 0$ has a non-zero term of order $\leq5$, 
		\item or the Taylor series of $\mathrm{Im}\,\phi-\gamma\cdot h$ at $\mathbf 0$ has a non-zero term of order $\leq4$. 
	\end{itemize}
\end{lem}
The proof of this lemma is rather delicate and will be postponed in Section~\ref{sec:lemmaproof} in order to keep a clearer exposition of the proof of Lemma~\ref{PROPMAIN}. However, using Lemma~\ref{lem:keylemma} we are able to finish this case. Indeed, if $\Gamma(u,v)=(\gamma(u,v),v)$ is the map given by the parametrized Morse lemma and if $f_1,f_2$ and $f_3$ are smooth functions such that $$\mathrm{Re}\,\phi(u,v)=\gamma(u,v)^2+f_1(v)\qquad\textrm{ and }\qquad\textrm{Im}\,\phi(u,v)=\gamma(u,v)f_2(u,v)+f_3(v),$$ then Lemma~\ref{lem:keylemma} implies that either $f_1(v)\sim_0 \alpha_p v^p$ with $p\leq 5$ or $f_3(v)\sim_0\beta_p v^p$ with $p\leq 4$. By the considerations above we may conclude that $\kappa = 9/8$ is possible. \qed

\subsection*{The case $J=1$ and dependence, $d\geq 3$} It remains to consider the case $J=1$, $d\geq 3$ and $\ell_1$ is a multiple of $\ell_{d+1}$. We shall deduce this case from the case $d=2$ using the following lemma. 
\begin{lem}
	\label{LEMD3} Let $\veps>0$, $\tau\in\RR$ and $z_3,\ldots,z_d\in\TT^{d-2}$. Consider the set
	\[A_{z_3,\ldots,z_d}(\tau,\veps)=\left\{(z_1,z_2)\in\TT^2\,:\, \Phi(z)\in Q(\tau,\veps)\right\}.\]
	Then, for every $w_3,\ldots,w_d\in\TT^{d-2}$, there exists a neighbourhood $\mathcal W \ni (w_3,\ldots,w_d)$ in $\TT^{d-2}$ such that, for all $(z_3,\ldots,z_d)\in\mathcal W$ we have
	\[A_{z_3,\ldots,z_d}(\tau,\veps)\subset A_{w_3,\ldots,w_d}(\tau,2\veps).\]
	\begin{proof}
		Assume that this is not the case. Then there exists a sequence $\big(z_1^{(k)},\ldots,z_d^{(k)}\big)$ in $\TT^d$ such that $z_j^{(k)}\to w_j$ for $3\leq j\leq d$ and
		\[\big(z_1^{(k)},z_2^{(k)}\big)\in A_{z_3^{(k)},\ldots,z_d^{(k)}}(\tau,\veps)\backslash A_{w_3,\ldots,w_d}(\tau,2\veps).\]
		Extracting a subsequence if necessary, we may assume that $z_1^{(k)}\to w_1$ and $z_2^{(k)}\to w_2$ for some $(w_1,w_2)\in\TT^2$. By the continuity of $\Phi$,this implies that $\Phi(w)\in \overline{Q(\tau,\veps)}\backslash Q(\tau,2\veps)$, which is a contradiction. 
	\end{proof}
\end{lem}

We now set
\[\Psi_{1,2}(z_1,z_2)=\Phi(z_1,z_2,1,\ldots,1).\]
Since $J(\Phi,w)=1$, we already know that $a_j>0$ for all $j=1,\ldots,d$ and hence the variables $z_1$ and $z_2$ both appear in the polynomial $\Psi_{1,2}$. Provided $\Psi_{1,2}(z)\neq \lambda_{1,2}(1-z_1z_2)$ for some $\lambda_{1,2}\in\RR^*$, we know from the case $d=2$ that there exists a neighbourhood $\mathcal V \ni (z_1,z_2)$ in $\TT^2$ and $C>0$ such that, for any $\tau\in\RR$ and every $\veps>0$, $$\mathbf m_2\big(\big\{(z_1,z_2)\in\mathcal V\,:\, \Phi(z_1,z_2,1,\ldots,1)\in Q(\tau,2\veps)\big\}\big)\leq C\veps^{\kappa}$$ with $\kappa=9/8$. By Lemma~\ref{LEMD3}, there exists a neighbourhood $\mathcal W \ni \mathbf 1$ in $\TT^{d-2}$ such that, for any $(z_3,\ldots,z_d)\in\mathcal W$, $$\big\{(z_1,z_2)\in\mathcal V\,:\, \Phi(z_1,\ldots,z_d)\in Q(\tau,\veps)\big\}\subset \big\{(z_1,z_2)\in\mathcal V\,:\,\ \Phi(z_1,z_2,1,\ldots,1)\in Q(\tau,2\veps)\big\}.$$ This yields $$\mathbf m_d\left(\left\{z\in\mathcal V\times\mathcal W\,:\,\ \Phi(z)\in Q(\tau,\veps)\right\}\right)\leq C\veps^{\kappa}.$$

So the result is proved except if, for every $j<k$, there exists some $\lambda_{j,k}>0$ such that 
\begin{equation}
	\label{eq:Lij} \Phi(1,\ldots,1,z_j,1,\ldots,1,z_k,1,\ldots,1)=\lambda_{j,k}(1-z_jz_k)=\lambda_{j,k}\big((1-z_j)+(1-z_k)-(1-z_j)(1-z_k)\big). 
\end{equation}
Comparing this with the expansion of $\Phi$ near $\mathbf 1$, we get
\[a_j=a_k=\lambda_{j,k},\qquad b_j=0,\qquad c_{j,k}=-\lambda_{j,k}.\]
Using \eqref{eq:relationsJ1}, we may conclude that $a_j=1/2$ and $c_{j,k}=-1/2$. In total this means that
\[\Phi(z_1,z_2,\ldots,z_d) = \frac{1}{2}\sum_{j=1}^d (1-z_j) - \frac{1}{2}\sum_{1 \leq j < k \leq d} (1-z_j)(1-z_k).\]
However,
\[\Phi(-1,-1,\ldots,-1) = \frac{2d-2d(d-1)}{2} = d(2-d)<0,\]
since $d \geq 3$. Hence \eqref{eq:Lij} is not possible for every $j<k$ and we are done. \qed

\section{Proof of Lemma~\ref{lem:keylemma}} \label{sec:lemmaproof} We intend to prove this result by contradiction. We require several tedious computations, which can be done either by hand or by a computer algebra system. We have used \texttt{Xcas}, and our file is available for download \cite{Xcas}. In the proof below, we will skip certain computations such as computing Taylor coefficients, simplifying algebraical expressions and solving simple equations. The proof consists of three steps, and in each step we refer to the lines in \cite{Xcas} where the computations are performed. 

The idea of the argument is rather easy. We assume that we may factorize $\mathrm{Re}\,\phi(u,v)$ and $\mathrm{Im}\,\phi(u,v)$ as (\ref{EQJ1-1}) and (\ref{EQJ1-2}) and we write 
\begin{eqnarray*}
	\gamma(u,v) &=&-2a_1a_2u+\gamma_{2,0} u^2+\gamma_{1,1} uv+\gamma_{0,2} v^2+\gamma_{3,0} u^3+\gamma_{2,1} u^2v+\gamma_{1,2} u^2v+\gamma_{0,3} v^3 \\
	& &+\gamma_{4,0}u^4+\gamma_{3,1} u^3v+\gamma_{2,2} u^2 v^2+\gamma_{1,3} u v^3+\gamma_{0,4} v^4 +o\big(|u|^5+|v|^5\big),\\
	h(u,v) &=&1+h_{1,0}u+h_{0,1}v+h_{2,0}u^2+h_{1,1}uv+h_{0,2}v^2+o\big(|u|^2+|v|^2\big). 
\end{eqnarray*}

We already know the first coefficients of $\gamma$ and $h$ by \eqref{EQJ1-3} and \eqref{EQJ1-4}. We will then compare the Taylor expansions of $\mathrm{Re}\,\phi(u,v)$ and $\mathrm{Im}\,\phi(u,v)$ obtained using \eqref{EQJ1-5} or using \eqref{EQJ1-1} and \eqref{EQJ1-2}. Looking at all coefficients of a given order, we will get first the value of the coefficients of the Taylor expansions of $\gamma$ and $h$ of a certain order and also an equation for $\mathrm{Im}(b_1)$, $\mathrm{Im}(b_2)$ and $\mathrm{Im}(c)$. 

At one point, we will have more equations than variables. These equations will have to be compatible, and will force $\Phi(z_1,z_2)=(1-z_1z_2)/2$, which is equivalent to saying $a_1=a_2=1/2$ and $\mathrm{Im}(b_1)=\mathrm{Im}(b_2)=\mathrm{Im}(c)=0$. This will imply the desired result.

\subsection*{Step 1} The goal of the first step is to show that if we have $a_1=a_2=1/2$, then we also have $\mathrm{Im}(b_1)=\mathrm{Im}(b_2)=\mathrm{Im}(c)=0$. In addition to this, we obtain some useful equations for the following steps. \cite[Lin.~1--14]{Xcas}

We begin by looking at the coefficients of $uv^2$ in the real part of $\Phi(u,v)$. Using on the one hand \eqref{EQJ1-5} and on the other hand \eqref{EQJ1-1} we conclude that
\[\gamma_{0,2} = \frac{-6a_1^3\mathrm{Im}(b_2)-6a_2^3\mathrm{Im}(b_1)+a_1a_2(a_1+a_2)\mathrm{Im}(c)}{8a_1a_2}.\]
We then obtain the first equation for $\mathrm{Im}(b_1)$, $\mathrm{Im}(b_2)$ and $\mathrm{Im}(c)$ by looking at the coefficients of $v^3$ in the real part: 
\begin{equation}
	\label{EQJ1-6} a_2^3\mathrm{Im}(b_1)-a_1^3\mathrm{Im}(b_2)+\frac{a_1a_2(a_2-a_1)}{2}\,\mathrm{Im}(c)=0. 
\end{equation}
Since we know the value of $\gamma_{0,2}$, we can get a second equation for $\mathrm{Im}(b_1)$, $\mathrm{Im}(b_2)$ and $\mathrm{Im}(c)$ by looking at the coefficients of $v^2$ in the imaginary part. Hence we get 
\begin{equation}
	\label{EQJ1-7} \frac{4a_1a_2^2-3a_2^2}{4a_1}\mathrm{Im}(b_1)+\frac{4a_1^2a_2-3a_1^2}{4a_2}\mathrm{Im}(b_2)+\frac{-8a_1a_2+a_1+a_2}{8}\,\mathrm{Im}(c)=0. 
\end{equation}
By the assumptions on $a_1$ and $a_2$, we know that $2(a_1+a_2)<3$ and hence we can solve \eqref{EQJ1-6} and \eqref{EQJ1-7} with respect to $\mathrm{Im}(b_1)$ and $\mathrm{Im}(b_2)$ to obtain
\[\mathrm{Im}(b_1)=\frac{a_1(2a_2^2+2a_1a_2+a_1-2a_2)}{2a_2^2(2a_1+2a_2-3)}\,\mathrm{Im}(c)\quad \textrm{ and } \quad \mathrm{Im}(b_2)=\frac{a_2(2a_1^2+2a_1a_2-2_1+a_2)}{2a_1^2(2a_1+2a_2-3)}\,\mathrm{Im}(c).\]
In particular, we may conclude that if $\mathrm{Im}(c)=0$, then we also have $\mathrm{Im}(b_1)=\mathrm{Im}(b_2)=0$. If we substitute these values into the expression for $\gamma_{0,2}$, we obtain
\[\gamma_{0,2} = \frac{-(a_1+a_2)^2}{2(2a_1+2a_2-3)}\,\mathrm{Im}(c).\]
Now, looking at the coefficient of $v^4$ in the real part shows that
\[(\gamma_{0,2})^2 = -\frac{a_1a_2(a_1+a_2)(a_1a_2^2+a_1^2a_2-a_1^2-a_2^2+a_1a_2)}{4},\]
and this yields our first expression for $\mathrm{Im}(c)^2$, 
\begin{equation}
	\label{EQJ1-9} \mathrm{Im}(c)^2=\frac{-a_1a_2(2a_1+2a_2-3)^2(a_1a_2^2+a_1^2a_2-a_1^2-a_2^2+a_1a_2)}{(a_1+a_2)^3}. 
\end{equation}
From \eqref{EQJ1-9}, it is evident that if $a_1 = a_2 = 1/2$, then $\mathrm{Im}(c)=0$. \qed

\subsection*{Step 2} In this step, we want to show that $a_1 = a_2$. Our first goal is to compute another equation to compare with \eqref{EQJ1-9}. \cite[Lin.~15--21]{Xcas}

We begin by successively looking at the coefficients of $u^2v$ in the real part, $uv$ in the imaginary part and $uv^3$ in the real part, to obtain 
\begin{eqnarray*}
	\gamma_{1,1}&=& \frac{a_1-a_2}{2}\,\mathrm{Im}(c),\\
	h_{0,1}&=&\frac{(a_1-a_2)(4a_1+4a_2-3)}{4a_1a_2(2a_1+2a_2-3)}\,\mathrm{Im}(c), \\
	\gamma_{0,3}&=&\frac{-{a_1}+{a_2}}{24 {a_1} {a_2} (2 {a_1}+2 {a_2}-3)}\times\\
	&&\big(20 {a_1}^2 {a_2}^{4}+40 {a_1}^{3} {a_2}^{3}+20 {a_1}^{4} {a_2}^{2}-12 {a_1} {a_2}^{4}-54 {a_1}^{2} {a_2}^{3}-54 {a_1}^{3} {a_2}^{2}\\
	&&\quad-12 {a_1}^{4} {a_2}+18 {a_1} {a_2}^{3}+18 {a_1}^{2} {a_2}^{2}+18 {a_1}^{3} {a_2}+3({a_1}+a_2)^2 \mathrm{Im}(c)^{2}\big). 
\end{eqnarray*}
Using these values, we will investigate the coefficient of $v^3$ in the imaginary part. This term depends indeed only on $\gamma_{1,1}$, $h_{0,1}$ and $\gamma_{0,3}$. Using the above expression, we obtain our second equation on $\mathrm{Im}(c)^2$: 
\begin{equation}
	\label{eq:Imc2} \frac{-3(a_1-a_2)(a_1+a_2)^2(a_1+a_2-1)}{4a_1a_2(2a_1+2a_2-3)^2}\mathrm{Im}(c)^2=\frac{-(a_1-a_2)(3a_1a_2^2+3a_1^2a_2-a_1^2-a_2^2-a_1a_2)}{4}. 
\end{equation}
At this stage, we have to consider several cases. Assuming that $a_1-a_2\neq 0$ and $a_1+a_2-1\neq 0$, we may compute 
\begin{equation}
	\label{EQJ1-8} \mathrm{Im}(c)^2=\frac{-a_1a_2(2a_1+2a_2-3)^2(3a_1a_2^2+3a_1^2a_2-a_1^2-a_2^2-a_1a_2)}{3(a_1+a_2)^2(a_1+a_2-1)}. 
\end{equation}
The only possibility is that the two values for $\mathrm{Im}(c)^2$ have to coincide. Equating \eqref{EQJ1-9} and \eqref{EQJ1-8} and simplifying, we obtain
\[\frac{a_1a_2(2a_1+2a_2-3)^2P(a_1,a_2)}{3(a_1+a_2)^2(a_1+a_2-1)}=0,\]
where
\[P(a_1,a_2)=2a_1^3+2a_2^3+a_1a_2^2+a_1^2a_2-3a_1^2-3a_2^2+3a_1a_2.\]
Since $2(a_1+a_2)<3$, the only possibility is that $P(a_1,a_2)$ vanishes somewhere in the domain
\[\mathcal{\Omega}=\{(a_1,a_2)\in (0,1)^2\,:\,\ a_2< a_1,\ a_2< 1-a_1\}.\]
We first look at what happens on the boundary, where we have 
\begin{eqnarray*}
	P(a_1,0)&=&2a_1^3-3a_1^2<0 \quad\text{ provided }a_1\in (0,1),\\
	P(a_1,a_1)&=&3a_1^2(2a_1-1)<0\quad\text{ provided }a_1\in (0,1/2),\\
	P(a_1,1-a_1)&=&-(2a_1-1)<0 \quad\text{ provided }a_1\in (1/2,1). 
\end{eqnarray*}
Hence, $P$ is negative on the boundary of $\mathcal \Omega$, except at $(1/2,1/2)$. Hence, if $P$ vanishes in $\Omega$, then it admits a critical point there. Now, we consider the system 
\[\left\{
\begin{array}{rcl}
	0=\frac{
	\partial P}{
	\partial a_1}(a_1,a_2)&=&6a_1^2+a_2^2+2a_1a_2-6a_1+3a_2, \\
	0=\frac{
	\partial P}{
	\partial a_2}(a_1,a_2)&=&a_1^2+6a_2^2+2a_1a_2+3a_1-6a_2. 
\end{array} \right.
\]
The solutions of this system are easily found to be at the intersection of two distinct ellipses. We cannot have more than two points of intersection and we have two trivial solutions, $(0,0)$ and $(1/3,1/3)$. Hence, none of the critical points of $P$ are inside $\mathcal \Omega$. Hence we get a contradiction, and we have finished this case.

Hence we must have $a_1 + a_2 = 1$ or $a_1 = a_2$. Let us now investigate the case $a_1+a_2=1$. Looking at \eqref{eq:Imc2}, this means that either $a_1 = a_2 = 1/2$ (and we are done) or
\[3a_1a_2^2+3a_1^2a_2-a_1^2-a_2^2-a_1a_2=0.\]
Taking into account that $a_2=1-a_1$, we get the equation $-4a_1^2+4a_1-1=0$ which admits the single solution $a_1=1/2$ and we get the same conclusion. Hence, the only remaining possibility is that $a_1 = a_2$. \qed

\subsection*{Step 3} It remains to deal with the case $a_1 = a_2 = a \in (0,1/2]$. We can no longer use \eqref{EQJ1-8} and need to find another equation for $\mathrm{Im}(c)^2$. \cite[Lin.~21--30]{Xcas}

We will be looking at the coefficient before $v^4$ in the imaginary part. By considering $\gamma \times h$, we see that this coefficient is equal to
\[\gamma_{0,4} + \gamma_{0,3}h_{0,1}+\gamma_{0,2}h_{0,2}.\]
Hence, it remains to compute $\gamma_{0,4}$ and $h_{0,2}$. First, we compute some auxiliary values. By looking at $u^3$ in the real part, $u^2$ in the imaginary part and $u^2v^2$ in the real part, respectively, we obtain 
\begin{eqnarray*}
	\gamma_{2,0} &=& - \frac{a(2a-1)}{3a-4}\,\mathrm{Im}(c), \\
	h_{1,0} &=& \frac{(2a-1)(4a-1)}{2a(4a-3)}\,\mathrm{Im}(c), \\
	\gamma_{1,2}&=&\frac{a(2a-1)\left(48a^4-72a^3+27a^2+4\mathrm{Im}(c)^2\right)}{4(4a-3)^2}. 
\end{eqnarray*}
Knowing these values, we look at the coefficients of $uv^4$ in the real part and $uv^2$ and the imaginary part, respectively, to obtain 
\begin{eqnarray*}
	\gamma_{0,4}&=&-a\mathrm{Im}(c)\frac{32a^5-96a^4+90a^3+12a\mathrm{Im}(c)^2-27a^2-6\mathrm{Im}(c)^2}{6(4a-3)^3},\\
	h_{0,2}&=&(-2a+1)\frac{128a^5-240a^4+144a^3+16a\mathrm{Im}(c)^2-27a^2-8\mathrm{Im}(c)^2}{8a(4a-3)^2}. 
\end{eqnarray*}
Finally, we investigate the coefficient of $v^4$ in the imaginary part to obtain the equation
\[a(2a-1)\frac{16a^4-32a^3+15a^2+4\mathrm{Im}(c)^2}{4(4a-3)^2}\,\mathrm{Im(c)}=0.\]
Now, if $\mathrm{Im}(c)=0$ and $a_1 = a_2 = a$, it follows at once from \eqref{EQJ1-9} that $a = 1/2$, since $a \in (0,1/2]$. Conversely, we divide away $\mathrm{Im}(c)$ and solve for $\mathrm{Im}(c)^2$ to obtain the equation
\[\mathrm{Im}(c)^2=-a^2\frac{(4a-3)(4a-5)}{4}.\]
Now, this has to be equal to \eqref{EQJ1-9}, and we find
\[-\frac{a(2a-1)(4a-3)^2}8=-\frac{a^2(4a-3)(4a-5)}{4}.\]
Here the only solutions are $a = 0$ and $a = 3/4$, neither of which belong to $(0,1/2]$. Hence the assumption $\mathrm{Im}(c)\neq0$ must be wrong and we conclude $a_1 = a_2 = 1/2$. \qed

\section{Remarks and Further Examples} \label{sec:discussion} If we look more closely at the map $\Phi$ defined in Lemma~\ref{lem:Cex3} (the negative part of Theorem \ref{thm:main}), then we may observe that these counterexamples all satisfy 
\begin{eqnarray*}
	\mathrm{Re}\,\Phi(e^{i\theta_1},\ldots,e^{i\theta_d})&=&\frac{\theta_1^2}2+o(\theta_1^2)\\
	\mathrm{Im}\,\Phi(e^{i\theta_1},\ldots,e^{i\theta_d})&=&-\theta_1+o(|\theta_1|). 
\end{eqnarray*}
Hence, $J\big(\Phi,\mathbf 1\big)=1$ and, using the terminology of the Section~\ref{sec:mainproof}, we have dependence. Our next results shows that we may also have non-compactness if $J(\Phi,w)=0$ for some $w\in\TT^d$.
\begin{thm}
	\label{prop:Cex5} There are polynomials $\varphi$ with unrestricted range, of any complex dimension $d\geq 2$ and of any complex degree $\geq4$ for which the corresponding composition operator $\mathcal{C}_\varphi$ is non-compact and such that they admit a minimal Bohr lift $\Phi$ satisfying $J(\Phi,w)=0$ for any $w\in\TT^d$ with $\mathrm{Re}\,\Phi(w)=0$. 
	\begin{proof}
		Let $\delta>0$ and define $$\Phi(z_1,z_2)=2(1-z_1)+(1-z_1)^2\big(1-\delta(1-z_2)-\delta(1-z_1)(1-z_2)\big).$$ Let $\varphi(s) = \Phi(p_1^{-s},p_2^{s})$. Clearly $\Phi$ is a minimal Bohr lift of $\varphi$. Then a computation shows that
		\[\mathrm{Re}\, \Phi(z_1,z_2)=2(1-\cos x)\left((1-\cos x)\big(1+2\delta(\sin{x}\sin{y}+(1-\cos{y})\cos{x})\big)+\delta(1-\cos y)\right).\]
		Clearly, for small enough $\delta>0$ this quantity is non-negative. Hence $\varphi\in \mathcal G$ and $J(\Phi,(1,1))=0$ because of the relation between $a_1$ and $b_1$. Considerations similar to those of Lemma~\ref{lem:Cex3} show that $\mathcal C_{\varphi}$ cannot be compact. To produce a counterexample with degree 3 and a bigger complex dimension $d$, we may simply replace $(1-z_2)$ by
		\[\frac{1}{d-1}\cdot\sum_{j=2}^d (1-z_j)\]
		in the definition of $\Phi$. The production of examples with degree $\geq 5$ is easier. We may just set
		\[\Phi(z) = (1-z_1) + \frac{1}{2}(1-z_1)^2 + \delta(1-z_1)^4P(z),\]
		where $P(z) = P(z_1,z_2,\ldots,z_d)$ is any polynomial. The proof follows now that of Lemma~\ref{lem:Cex3}. 
	\end{proof}
\end{thm}
The reason the first counterexample in Theorem~\ref{prop:Cex5} works is that we have a cancellation of the term $(1-\cos{x})\sin{x}\sin{y}$. It seems difficult to obtain the same cancellation if we restrict ourself to degree $3$ and require $J=0$.
\begin{que}
	Is it possible to construct a counterexample of degree $3$ with $J=0$? 
\end{que}

An answer to the question would in a certain sense improve the optimality of Lemma~\ref{PROPMAIN}, but it would not yield the complete answer to which Dirichlet polynomials in $\mathcal{G}$ induce compact composition operators. Indeed, the natural next point of investigation would be this: What happens when the ``quartic form'' is degenerate?

In this case, terms of degree $5$ also have to disappear. This follows by the mapping properties and the argument is identical to the one used to show that degree $3$ terms disappear in the case $J=0$ given above. Hence we are reduced to studying a ``sextic form''. 

Our counterexamples can be modified to work in this case, but they now have degree $6$ and $7$. Degree $\leq 3$ will also easily reduce to the case of Theorem~\ref{thm:polynomials} in the same manner as $J=0$ did for degree $\leq 2$. However, the cases with degree $4$ and $5$ would need further investigation. Even if we could solve this case, we would need to investigate the case when the ``sextic form'' is degenerate and this leads to the ``octic form'' and so on.
\begin{rem}
	The previous counterexample shows that we cannot deduce Theorem \ref{thm:polynomials} from Theorem \ref{thm:main}. Indeed, it is easy to construct symbols $\varphi\in \mathcal G$ which may be written $\varphi(s)=\sum_{j=1}^d P_j(p_j^{-s})$ and such that $J(\Phi,\mathbf 1)=0$ for $\Phi$ a minimal Bohr lift of $\varphi$. Indeed, we may consider $$P_j(z)=(1-z)+\frac12(1-z)^2+\delta(1-z)^4 Q_j(z)$$ where $Q_j$ is an arbitrary polynomial and $\delta>0$ is sufficiently small. Then $\mathcal C_{\varphi}$ is compact by Theorem~\ref{thm:polynomials} if $d\geq2$ but this cannot be deduced from Lemma~\ref{PROPMAIN}. 
\end{rem}

This construction can be generalized to show that Theorem~\ref{thm:polynomials} can handle a variety of different interesting cases not covered by Theorem~\ref{thm:main}. In fact, given any $d$ positive integers $k_j$, we may find a polynomial $\Phi(z) = \sum_{j=1}^d \Phi_j(z_j)$ which is a minimal Bohr lift of some $\varphi \in \mathcal{G}$, with $\mathrm{Re}\,\Phi(z_1,\ldots,z_d)=0$ if and only if $z=\mathbf 1$ and here we have the expansion 
\begin{align*}
	\mathrm{Re}\,\Phi\big(e^{i\theta_1},\,\ldots,\, e^{i\theta_d}\big) &= \sum_{j=1}^d \theta_j^{2k_j} + o\left(\sum_{j=1}^d \theta_j^{2k_j}\right), \\
	\mathrm{Im}\,\Phi\big(e^{i\theta_1},\,\ldots,\, e^{i\theta_d}\big) &= \sum_{j=1}^d a_1^{(j)}\theta_j + o\left(\sum_{j=1}^d |\theta_j|\right). 
\end{align*}
As remarked upon in the proof of Theorem~\ref{thm:polynomials}, we must have $a_1^{(j)}>0$. The construction of such a polynomial is immediate from our next result.
\begin{lem}
	\label{lem:1dimpoly} There is a polynomial $\Phi:\CC\to\CC$ which satisfy $\mathrm{Re}\,\Phi\big(e^{ix}\big) = (1-\cos{x})^k$, for any $k \in \mathbb{N}$. 
	\begin{proof}
		Fix $N$ with $k \leq 2N$, and for real numbers $a_n$ and $b_n$ consider
		\[\Phi(z) = \sum_{n=1}^N \frac{(-1)^{n-1}}{2^n}\big(a_n(1-z)^{2n-1} - b_n(1-z)^{2n}\big).\]
		Our first goal is to expand the real part of $\Phi\big(e^{ix}\big)$ as degree $2N$ polynomial in $(1-\cos{x})$ with no constant term. To this end, we compute
		\[\big(1-e^{ix}\big)^{2n-1} = e^{i(2n-1)x/2}\big(e^{-ix/2}-e^{ix/2}\big)^{2n-1} = e^{i(2n-1)x/2}2^{2n-1}(-1)^n i \sin^{2n-1}\left(\frac{x}{2}\right).\]
		We use $2\sin^{2}(x/2)=1-\cos{x}$, and obtain 
		\begin{equation}
			\label{eq:reOdd} \mathrm{Re}\,\big(1-e^{ix}\big)^{2n-1} = 2^{2n-1}(-1)^{n-1} \sin^{2n-1}\left(\frac{x}{2}\right)\sin\left(nx-\frac{x}{2}\right) = (-1)^{n-1} 2^n (1-\cos{x})^n \frac{\sin\left(nx-\frac{x}{2}\right)}{2\sin\left(\frac{x}{2}\right)}. 
		\end{equation}
		Similarly, we obtain 
		\begin{equation}
			\label{eq:reEven} \mathrm{Re}\,\big(1-e^{ix}\big)^{2n} = 2^{2n}(-1)^n \sin^{2n}\left(\frac{x}{2}\right)\cos(nx) = (-1)^n 2^n (1-\cos{x})^n \cos{(nx)}. 
		\end{equation}
		To continue the computations, we introduce the Chebyshev polynomials
		\[ U_n(y) = \sum_{j=0}^n (-2)^j\frac{(n+j+1)!}{(n-j)!(2j+1)!}(1-y)^j, \qquad \text{and} \qquad T_n(y) = n\sum_{j=0}^n (-2)^j \frac{(n+j-1)!}{(n-j)!(2j)!}(1-y)^j.\]
		The Chebyshev polynomials are relevant due to the formulas $\sin{nx} = \sin(x)U_{n-1}(\cos{x})$ and $\cos{nx} = T_n(\cos{x})$. We record the following coefficients. 
		\begin{align*}
			u_{n-2}^{(n)} = (-2)^{n-1}\,\left(-\frac{(2n-1)(n-2)}{2}\right),\qquad\quad u_{n-1}^{(n)} &= (-2)^{n-1}\,(2n),\qquad &u_n^{(n)} = (-2)^{n-1}\,(-2), \\
			t_{n-1}^{(n)} &= (-2)^{n-1}\,(n),\qquad &t_{n}^{(n)} = (-2)^{n-1}\,(-1). 
		\end{align*}
		Now, we rewrite \eqref{eq:reEven} as
		\[\mathrm{Re}\,\big(1-e^{ix}\big)^{2n} = (-1)^n2^n(1-\cos{x})^n T_n(\cos{x}),\]
		which is then clearly a degree $2n$ polynomial in $(1-\cos{x})$ with no constant term. For \eqref{eq:reOdd} we have to work a bit more, so we first compute
		\[\frac{\sin\left(nx-\frac{x}{2}\right)}{2\sin\left(\frac{x}{2}\right)} = \frac{\sin(nx)\cos\left(\frac{x}{2}\right)-\cos(nx)\sin\left(\frac{x}{2}\right)}{2\sin\left(\frac{x}{2}\right)}=\cos^2\left(\frac{x}{2}\right)U_{n-1}(\cos{x})-\frac{T_n(\cos{x})}{2},\]
		which implies that we may rewrite \eqref{eq:reOdd} as
		\[\mathrm{Re}\,\big(1-e^{ix}\big)^{2n-1} = (-1)^{n-1} 2^n (1-\cos{x})^n \left(\left(1-\frac{1-\cos{x}}{2}\right)U_{n-1}(\cos{x})-\frac{T_n(\cos{x})}{2}\right).\]
		Again we observe that this is a polynomial of degree $2n$ in $(1-\cos{x})$ with no constant term. In total, we have
		\[\mathrm{Re}\,\Phi\big(1-e^{ix}\big) = \sum_{m=1}^{2N} c_m(1-\cos{x})^{m} = \sum_{n=1}^N \left(a_n P_n(1-\cos{x})+ b_n Q_n(1-\cos{x})\right),\]
		where 
		\begin{align*}
			P_n(y) &= \sum_{j=0}^{n} d_j^{(n)} y^{n+j} = y^n\left(\left(1-\frac{y}{2}\right)U_{n-1}(1-y)-\frac{T_n(1-y)}{2}\right), \\
			Q_n(y) &= \sum_{j=0}^{n} e_j^{(n)} y^{n+j} = y^nT_n(1-y). 
		\end{align*}
		Given any choice of $c_m$ (for instance $c_m = 0$ for $m\neq k$ and $c_k=1$), we now have $2N$ linear equations and $2N$ unknowns, $a_n$ and $b_n$ for $1\leq n\leq N$. We will now show that this system can always be solved. 
		
		We first observe that $a_n$ and $b_n$ only have an effect on $c_m$ when $n \leq m \leq 2n$. Ordering the unknowns as $a_N,b_N,a_{N-1},b_{N-1},\ldots,a_1,b_1$ and the datas as $c_{2N},c_{2N-1},\ldots,c_1$, this means that the matrix of our system can be written in upper triangular block form, where the blocks on the diagonal are
		\[
		\begin{pmatrix}
			d_{n-1}^{(n)} & e_{n-1}^{(n)} \\
			d_{n}^{(n)} & e_{n}^{(n)} \\
		\end{pmatrix},\qquad n=N,N-1,\,\ldots\,,1.
		\]
		We know that $e_{n-1}^{(n)} = t_{n-1}^{(n)}$ and $e_{n}^{(n)} = t_n^{(n)}$, which we recorded above. It is now easy to verify that 
		\begin{align*}
			d_{n-1}^{(n)} &= u_{n-1}^{(n)} - \frac{u_{n-2}^{(n)}}{2} - \frac{t_{n-1}^{(n)}}{2} = (-2)^{n-1}\left(\frac{3n}{2} + \frac{(2n-1)(n-2)}{4}\right), \\
			d_{n}^{(n)} &= u_{n}^{(n)} - \frac{u_{n-1}^{(n)}}{2} - \frac{t_n^{(n)}}{2} = (-2)^{n-1}\left(-\frac{3}{2}-n\right). 
		\end{align*}
		Hence we are reduced to considering the equation
		\[0 = \frac{d_{n-1}^{(n)}e_n^{(n)}-d_{n}^{(n)}e_{n-1}^{(n)}}{4^{n-1}} = \left(\frac{3n}{2} + \frac{(2n-1)(n-2)}{4}\right)(-1)-\left(-\frac{3}{2}-n\right)n =n^2-\frac{(2n-1)(n-2)}{4},\]
		which has no integer solutions, and we are done. 
	\end{proof}
\end{lem}

The construction of $\Phi$ with specific expansion facilitated by Lemma~\ref{lem:1dimpoly} will be used in the next section to prove Corollary~\ref{cor:schatten}.

\section{Approximation numbers}\label{sec:approximationnumber} In this section, we consider only the case $c_0=0$. We intend to estimate the decay of $a_n(\mathcal C_\varphi)$ for maps $\varphi$ which are, in a certain sense, regular at their boundary points. For this we need as previously a careful inspection of the behaviour of the Bohr lift $\Phi$ near these boundary points. 
\begin{defn}
	Suppose that $\varphi(s)=c_1+\sum_{n=2}^N c_n n^{-s} \in \mathcal{G}$, and that $\varphi$ has complex dimension $d$ and unrestricted range. Let $\Phi$ be a minimal Bohr lift of $\varphi$ and let $w\in\TT^d$ be such that $\mathrm{Re}\,\Phi (w)=0$. We say that $\varphi$ is \emph{boundary regular} at $w$ if there exist independent linear forms $\ell_1,\,\ldots,\,\ell_d$ on $\mathbb{C}^d$, even integers $k_1\geq k_2\geq\ldots\geq k_d$ and real numbers $b_1,\,\ldots,\,b_d,\,\tau$ with $b_1\neq 0$ such that 
	\begin{eqnarray}
		\mathrm{Re}\,\Phi\big(e^{i\theta_1}w_1,\,\ldots,\,e^{i\theta_d}w_d\big)&=&\ell_1(\theta)^{k_1}+\ldots+\ell_d(\theta)^{k_d}+\sum_{j=1}^d o\left(\ell_j^{k_j}(\theta)\right)\label{eq:apboundaryregular1}\\
		\mathrm{Im}\,\Phi\big(e^{i\theta_1}w_1,\,\ldots,\,e^{i\theta_d}w_d\big)&=&\tau+b_1\ell_1(\theta)+\ldots+b_d\ell_d(\theta)+o\left(\sum_{j=1}^ d|\ell_j(\theta)|\right).\label{eq:apboundaryregular2} 
	\end{eqnarray}
	 We define the \emph{compactness index} of $\varphi$ at $w$ as $$\eta_{\varphi,w}=\left(\sum_{j=2}^d \frac1{k_j}\right)\times\frac{k_1}{2(k_1-1)}.$$ If every boundary point is boundary regular, we say that $\varphi$ is \emph{boundary regular}.
\end{defn}

The proof of Theorem \ref{thm:polynomials} then shows that given a boundary regular map $\varphi$, the composition operator $\mathcal C_\varphi$ is compact if and only if $d\geq2$. We shall now assume that there is only one point $w \in \mathbb{T}^d$ such that $\mathrm{Re}\,\Phi(z)=0$. In this case, we let the \emph{compactness index} of $\varphi$ be $\eta_\varphi:=\eta_{\varphi,w}$. 

The main theorem of this section now reads. 
\begin{thm}
	\label{thm:approximationnumber} Let $\varphi(s)=c_1+\sum_{n=2}^N c_nn^{-s}\in\mathcal G$ have unrestricted range and complex dimension $d$. Let $\Phi$ be a minimal Bohr lift and assume that there exists a unique $w\in\TT^d$ such that $\mathrm{Re}\,\Phi(w)=0$. Suppose moreover that $\varphi$ is boundary regular at $w$. Then $$\left(\frac{1}{n}\right)^{\eta_\varphi}\ll a_n(\mathcal C_\varphi)\ll \left(\frac{\log n}n\right)^{\eta_\varphi}.$$ 
\end{thm}
This statement may be applied to several cases. 
\begin{cor}
	Let $\varphi(s)=c_1+\sum_{n=2}^N c_nn^{-s}\in\mathcal G$ have unrestricted range and complex dimension $d$. Let $\Phi$ be a minimal Bohr lift of $\varphi$ and assume that there exists a unique $w\in\TT^d$ such that $\mathrm{Re}\,\Phi(w)=0$ and that $J(\Phi,w)=d$. Then $$\left(\frac{1}{n}\right)^{(d-1)/2}\ll a_n(\mathcal C_\varphi)\ll \left(\frac{\log n}n\right)^{(d-1)/2}.$$ 
	\begin{proof}
		Under these assumptions, $\varphi$ is boundary regular at $w$ with $k_1=\ldots=k_d=2$. 
	\end{proof}
\end{cor}

In particular, this corollary covers the result of Queff\'{e}lec and Seip for linear symbols \eqref{eq:linearcase}, as well as the map $\varphi_1$ given in \eqref{eq:simplemixedcase}. We may also apply Theorem \ref{thm:approximationnumber} to the maps considered in Theorem \ref{thm:polynomials}. In this case, one has simply $\ell_j(\theta)=\theta_j$ (up to a reordering of the terms). 

Another interesting application of Theorem \ref{thm:approximationnumber} is that we may distinguish the Schatten classes of bounded linear operators on $\mathcal{H}^2$ using composition operators, as mentioned in the introduction. 
\begin{proof}
	[Proof of Corollary~\ref{cor:schatten}] Let $p'<q'$ and $\veps>0$ be such that $p\leq p'/2$ and $\left(\frac12+\veps\right)q'\leq q$. Then let $d\geq 2$ and $k\geq 2$ even such that $$p'<\frac{d-1}{k}<q'\textrm{ and }\frac12<\frac{k}{2(k-1)}<\frac12+\veps.$$ By Lemma \ref{lem:1dimpoly}, we know that there exists a boundary regular polynomial $\Phi:\TT^d\to\mathbb C_0$ such that $\mathrm{Re}\,\Phi(w)=0$ if and only if $w=\mathbf 1$ and $$\Phi\big(e^{i\theta_1},\ldots,e^{i\theta_d}\big)=\theta_1^{k}+\ldots+\theta_d^k+o\big(\theta_1^{k}\big)+\ldots+o\big(\theta_d^k\big).$$ Letting $\varphi\in\mathcal G$ any map such that $\Phi$ is a minimal Bohr lift of $\varphi$, we immediately get
	\[\left(\frac 1n\right)^{\frac{d-1}k\times \frac{k}{2(k-1)}}\ll a_n(\mathcal C_\varphi)\ll \left(\frac {\log n}n\right)^{\frac{d-1}k\times \frac{k}{2(k-1)}},\]
	which completes the proof. 
\end{proof}
Theorem \ref{thm:approximationnumber} may be also applied to many other maps. We will consider here the map $\varphi_2$ given in \eqref{eq:simplemixedcase}. Its boundary regularity is different than that of $\varphi_1$, and hence the degree of compactness is also different. 
\begin{ex}
	Let $\varphi_2(s)=13/2-4\cdot 2^{-s}-4\cdot 3^{-s}+2\cdot 6^{-s}$ as in \eqref{eq:simplemixedcase} and let $\Phi$ be its minimal Bohr lift. It can be shown that $\mathrm{Re}\,\Phi(w)=0$ for $w\in\TT^2$ if and only if $w=(1,1)$, and 
	\begin{eqnarray*}
		\mathrm{Re}\,\Phi\big(e^{i\theta_1},e^{i\theta_2}\big)&=&\ell_1(\theta)^4+\ell_2(\theta)^2+o\big(\ell_1^4(\theta)\big)+o\big(\ell_2^2(\theta)\big),\\
		\mathrm{Im}\,\Phi\big(e^{i\theta_1},e^{i\theta_2}\big)&=&-2\ell_1(\theta)+o\big(|\ell_1(\theta)|+|\ell_2(\theta)|\big), 
	\end{eqnarray*}
	where $\ell_1(\theta)=\theta_1+\theta_2$ and $\ell_2(\theta)=\theta_1-\theta_2$. Hence $\eta_{\varphi_2} = (1/2)\times(4/6) = 1/3$. 
\end{ex}
The remaining part of this section is devoted to the proof of Theorem \ref{thm:approximationnumber}. We use the scheme introduced by Queff\'elec and Seip in \cite{QS14} in the context of Dirichlet series (see also \cite{QS15a} for similar works on the classical Hardy space of the disk). Their method is based on Carleson measures, interpolation sequences and model spaces. In Subsection \ref{sec:aptools}, we survey these tools and give a couple of lemmas. Subsection \ref{sec:apupperbound} is devoted to the proof of the upper bound, in a more general context, whereas Subsection \ref{sec:aplowerbound} will be devoted to the lower bound.

\subsection{Tools} \label{sec:aptools} 
\subsubsection*{The Hyperbolic Metric} The pseudo-hyperbolic metric on the half-plane $\mathbb C_0$ is defined by $$\rho(z,w)=\left|\frac{z-w}{z+\overline{w}}\right|=\frac{1-e^{-d(z,w)}}{1+e^{d(z,w)}}$$ where $d(z,w)$ is the hyperbolic distance between two points $z$ and $w$ in $\mathbb C_0$. The hyperbolic length of a curve $\Gamma\subset\mathbb C_0$ is given by the integral $$L_p(\Gamma)=\int_{\Gamma}\frac{|dz|}{\mathrm{Re}\, z}.$$

\subsubsection*{Carleson Measures and Interpolating Sequences} Let $H$ be a Hilbert space of functions defined on some measurable set $\Omega$ in $\CC$. A non-negative Borel measure $\mu$ on $\Omega$ is a \emph{Carleson measure} for $H$ if there exists some constant $C>0$ such that $$\int_{\Omega}|f(z)|^2d\mu(z)\leq C\|f\|_H^2,$$ for every $f$ in $H$. The smallest possible $C$ will be called the \emph{Carleson norm} of $\mu$ with respect to $H$ and will be denoted by $\|\mu\|_{\mathcal C,H}$. 

We also assume that the linear point evaluation is bounded at any $z\in\Omega$. Then $H$ admits a reproducing kernel $K_z^H\in H$ for any $z\in\Omega$ which satisfies $f(z)=\left\langle f,K_z^H\right\rangle$ for every $f\in H$. We then say that a sequence $Z=(z_m)$ of distinct points in $\Omega$ is a \emph{Carleson sequence} for $H$ if the measure $$\mu_{Z,H}:=\sum_m \left\|K_{z_m}^H\right\|^{-2}_H\delta_{z_m}$$ is a \emph{Carleson measure} for $H$.

We say that a sequence $Z=(z_m)$ of distinct points in $\Omega$ is an \emph{interpolating sequence} for $H$ if the interpolation problem $f(z_m)=a_m$ has a solution $f\in H$ whenever the admissibility condition $$\sum_m |a_m|^2 \left\|K_{z_m}^H\right\|^{-2}_H<\infty$$
is satisfied. By the open mapping theorem, if $Z$ is an interpolating sequence for $H$, there is a constant $C>0$ such that we can solve $f(z_m)=a_m$ with $f$ satisfying $$\|f\|_H\leq C\left(\sum_m |a_m|^2 \left\|K_{z_m}^H\right\|^{-2}_H\right)^{1/2}.$$ The smallest constant $C$ with this property will be called the \emph{constant of interpolation} of $Z$ and will be denoted by $M_H(Z)$.

We shall consider the two spaces $H=\mathcal H^2$ and $H=H^2\left(\mathbb T^d\right)$. Then we have, respectively, $\Omega=\CC_{1/2}$ and $\Omega=\DD^d$, and moreover
\[\big\|K_s^{\mathcal H^2}\big\|^{-2}=[\zeta(2\mathrm{Re}\,s)]^{-1} \qquad \text{and} \qquad \big\|K_z^{H^2\left(\TT^d\right)}\big\|^{-2}=\prod_{j=1}^d (1-|z_j|^2).\]
We will need the three following lemmas.
\begin{lem}
	\label{lem:apcarleson1} Let $\mu$ be a Borel measure on $\overline{\CC_0}$, let $\sigma\in(0,1)$ and $R>0$. Assume that $\mu$ is supported on the rectangle $0\leq\mathrm{Re}\, s\leq \sigma$, $|\mathrm{Im}\, s|\leq R.$ Then $$\|\mu\|_{\mathcal C,\mathcal H^2}\ll_R \sup_{\veps>0,\ \tau\in\RR}\frac{\mu\big(Q(\tau,\veps)\big)}{\veps}\leq 2\sup_{\veps\in(0,\sigma),\ \tau\in\RR}\frac{\mu\big(Q(\tau,\veps)\big)}{\veps}.$$ 
	\begin{proof}
		The first inequality is Lemma 2.3 in \cite{QS14} (the involved constant does not depend on $\sigma\in(0,1)$). The second follows from the inequality $$\sup_{\tau\in \RR}\frac{\mu\big(Q(\tau,2^{k+1}\sigma)\big)}{2^{k+1}\sigma}\leq\sup_{\tau\in \RR}\frac{\mu\big(Q(\tau,2^{k}\sigma)\big)}{2^{k}\sigma},$$ valid for any $k\geq 0$. Indeed, for any $\tau\in\mathbb R$ and any $k\geq 0$, we may find $\tau_1,\tau_2\in\mathbb R$ such that $$\mu\big(Q(\tau,2^{k+1}\sigma)\big)=\mu\big(Q(\tau_1,2^{k}\sigma)\big)+\mu\big(Q(\tau_2,2^{k}\sigma)\big),$$ since the support of $\mu$ is contained in $0\leq\mathrm{Re}\, s\leq \sigma$. 
	\end{proof}
\end{lem}

\begin{lem}
	\label{lem:apinterpolation1} Let $\nu>0$. There exists $C>0$ such that, for any $\delta\in(0,1/\nu)$,  $M_{\mathcal H^2}(S_\delta)\leq C$ where $S_\delta=(s_m)_{m=1}^{1/\delta}$ with $s_m=\frac12+\nu\delta+im\delta$. 
	\begin{proof}
		The proof of this lemma can be found in \cite[Sec~8.2]{QS14}. 
	\end{proof}
\end{lem}

\begin{lem}
	\label{lem:apinterpolation2} Let $C_1,C_2>0$. There exists $D>0$ such that for any $\delta>0$ and any (finite) sequence 
	\[Z=\big(Z(\alpha)\big)=\big((1-\rho_1(\alpha))e^{i\theta_1(\alpha)},\,\ldots,\,(1-\rho_d(\alpha))e^{i\theta_d(\alpha)}\big)\]
	in $\mathbb{D}^d$ satisfying
	\begin{itemize}
		\item $\sup_{j=1,\ldots,d}|\theta_j(\alpha)-\theta_j(\beta)|\geq C_1\delta$, when $\alpha\neq \beta$,
		\item $\rho_j(\alpha)\leq C_2\delta$, for any $\alpha$ and $j=1,\,\ldots,\,d$,
	\end{itemize}
	we have $\left\|\mu_{Z,H^2(\TT^d)}\right\|_{\mathcal C,H^2(\TT^d)}\leq D$. 
	\begin{proof}
		To each point $Z(\alpha)$, we associate a rectangle $R_\alpha$ on the distinguish boundary $\TT^d$ centered at $$\left(\frac{z_1(\alpha)}{|z_1(\alpha)|},\,\ldots,\,\frac{z_d(\alpha)}{|z_d(\alpha)|}\right),$$ 
		with side lengths $2(1-|z_1(\alpha)|),\,\ldots,\,2(1-|z_d(\alpha)|)$. By Chang's characterization of Carleson measures on the polydisc (see \cite{BCL87} or \cite{Ch79}), it is enough to show that we for all open sets $\mathcal U$ of $\TT^d$ have 
		\[\sum_{R_\alpha\subset\mathcal U}\mathbf{m}_d(R_\alpha)\leq D\mathbf{m}_d(\mathcal U).\]
		If $R$ is some rectangle in $\TT^d$ and $\lambda>0$, denote by $\lambda R$ the rectangle with the same center and side lengths multiplied by $\lambda$. Then our assumptions on $Z$ imply that there exists some $\lambda\in(0,1)$ depending only on $C_1$ and $C_2$ such that the rectangles $R_\alpha$ are pairwise disjoint. Thus 
		\[\sum_{R_\alpha\subset\mathcal U}\mathbf m_d(\mathcal U)\leq\sum_{R_\alpha\subset\mathcal U}\frac{1}{\lambda^d}\mathbf m_d(\lambda R_\alpha)\\
			\leq\frac1{\lambda^d}\mathbf m_d\left(\bigcup_{R_\alpha\subset\mathcal U}\lambda R_\alpha\right)\\
			\leq\frac1{\lambda^d}\mathbf m_d(\mathcal U),\]
			which completes the proof with $D = 1/\lambda^d$.
	\end{proof}
\end{lem}

\subsubsection*{The Queff\'{e}lec--Seip Method.} We have to introduce additional conventions. For $\varphi\in\mathcal G$ and $\Omega$ a compact subset of $\CC_0$, we denote by $\mu_{\varphi,\Omega}$ the non-negative Borel measure on $\overline{\CC_0}$ defined by 
$$\mu_{\varphi,\Omega}(E):= \mathbf{m}_\infty \left(\left\{z \in \mathbb{T}^\infty \,:\, \Phi(z) \in E\backslash \Omega\right\}\right).$$
Next, assume that $\varphi$ has complex dimension $d$ and Bohr lift $\Phi:\CC^d\to\CC$. Let $S=(s_m)$ be a sequence of $n$ points in $\CC_{1/2}$ and let $Z$ be a finite sequence of points in $\DD^d$ such that $\Phi(Z)=S-\frac12$. We set 
$$ N_{\Phi}(s_m;Z):=\sum_{z\in Z\cap \Phi^{-1}(s_m-1/2)}\big\|K_z^{H^2(\TT^d)}\big\|^{-2}.$$
We state Theorem 4.1 of \cite{QS14} as the forthcoming lemma (we have modified it slightly to take into account our normalization). 
\begin{lem}
	\label{lem:apqsmethod} Let $\varphi(s)=\sum_{n=1}^{\infty}c_nn^{-s}\in\mathcal G$ such that $\varphi(\CC_0)$ is bounded. 
	\begin{itemize}
		\item[(a)] Let $\sigma>0$ and $\Omega$ be a compact subset of $\overline{\CC_\sigma}$. Let $B$ be a Blaschke product on $\CC_0$ whose zeros lie in $\Omega$. Then $$a_n(\mathcal C_\varphi)\leq \left(\sup_{s\in\Omega}|B(s)|^2\zeta(1+2\sigma)+\sup_{\veps>0,\tau\in\mathbb R}\frac{\mu_{\varphi,\Omega}(Q(\tau,\veps))}{\veps}\right)^{1/2}.$$ 
		\item[(b)] Assume that $\varphi$ has complex dimension $d$. Let $S$ and $Z$ be finite sets in respectively $\CC_{1/2}$ and $\DD^d$ such that $\Phi(Z)=S-\frac12$. Then $$a_n(\mathcal C_\varphi)\geq \left[M_{\mathcal H^2}(S)\right]^{-1}\left\|\mu_{Z,H^2(\TT^d)}\right\|^{-1/2}_{\mathcal C,H^2(\TT^d)}\inf_m \left[N_\Phi(s_m;Z)\zeta(2\mathrm{Re}\,s_m)\right]^{1/2}.$$ 
	\end{itemize}
\end{lem}

\medskip \subsection{The Upper Bound} \label{sec:apupperbound} \hspace{1cm} \medskip 

Let $\varphi(s)=\sum_{n=1}^{\infty}c_nn^{-s}\in\mathcal G$ and suppose that $\varphi(\CC_0)$ bounded. By Lemma \ref{lem:QS}, $\mathcal C_\varphi$ is compact if and only if $\mu_\varphi\big(Q(\tau,\veps)\big)=o(\veps)$ uniformly in $\tau\in\RR$. We are planning to get an upper bound of $a_n(\mathcal C_\varphi)$ depending on the behaviour of $\sup_{\tau\in\mathbb R} \mu_\varphi\big(Q(\tau,\veps)\big)$ with respect to $\veps$ and on the size of the image of $\varphi$ near a boundary point. 

Thus, let $\Phi$ be a Bohr lift of $\varphi$. We define $\kappa_{\varphi}$ as the infimum of those $\kappa\geq 1$ such that there exists a constant $C>0$ such that, for every $\tau\in\RR$ and every $\veps>0$, $$ \mathbf{m}_\infty \left(\left\{z \in \mathbb T^\infty \, : \, \Phi(z) \in Q(\tau,\,\varepsilon) \right\} \right)\leq C \varepsilon^{\kappa}.$$ Assume now that there exists a unique $w\in\TT^\infty$ such that $\mathrm{Re}\,\Phi(w)=0$ and write $\Phi(w)=i\tau$. Let $\omega_\varphi$ be the infimum of the positive $\omega$ such that, for any $s\in\CC_0$, $$|\textrm{Im}\,\varphi(s)-\tau|^\omega\leq C\left(\mathrm{Re}\,\varphi(s)-\frac 12\right).$$
\begin{thm}
	\label{thm:apupperbound} Let $\varphi(s)=\sum_{n=1}^{\infty}c_nn^{-s}\in\mathcal G$ with $\varphi(\CC_0)$ bounded, let $\Phi$ be a Bohr lift of $\varphi$ and assume that there is a unique $w\in\TT^\infty$ such that $\mathrm{Re}\,\Phi(w)=0$. Then 
	\[a_n(\mathcal{C}_\varphi) \ll \begin{cases}
		\exp\left(-\lambda n^{-1/2}\right)	& \text{if } \omega_\varphi \leq 1,\\
		\left(\frac{\log{n}}{n}\right)^{(\kappa_\varphi-1)\times\frac{\omega_\varphi}{2(\omega_\varphi-1)}} & \text{if } \omega_\varphi > 1.
	\end{cases}\]
	Here $\lambda$ is some positive constant depending on $\varphi$.
\end{thm}

This theorem illustrates the following general principle for composition operators (valid beyond $\mathcal H^2$): The more restricted the image of the symbol is, the more compact the associated composition operator is. In particular, the case $\omega_{\varphi}=1$ (the range of $\varphi$ is contained in an angle) is reminiscent from \cite[Thm.~1.2]{QS15a} where a similar result was obtained for composition operators on $H^2(\DD)$. 

Before we embark upon the proof of Theorem~\ref{thm:apupperbound}, we first employ it to deduce the upper bound of Theorem \ref{thm:approximationnumber}.
\begin{proof}[Final part in the proof of the upper bound of Theorem \ref{thm:approximationnumber}] 
	Suppose that $\varphi \in \mathcal{G}$ is a boundary regular Dirichlet polynomial, and assume that $\mathrm{Re}\,\Phi(\mathbf 1)=0$. We write 
	\begin{eqnarray*}
		\mathrm{Re}\,\Phi\big(e^{i\theta_1},\,\ldots,\,e^{i\theta_d}\big)&=&\ell_1(\theta)^{k_1}+\cdots+\ell_d(\theta)^{k_d}+\sum_{j=1}^do\big(\ell_j^{k_j}(\theta)\big),\\
		\mathrm{Im}\,\Phi\big(e^{i\theta_1},\,\ldots,\,e^{i\theta_d}\big)&=&\tau+b_1\ell_1(\theta)+\cdots+b_d\ell_d(\theta)+o\left(\sum_{j=1}^d |\ell_j(\theta)|\right), 
	\end{eqnarray*}
	with $k_1\geq\ldots\geq k_d$ and $b_1\neq 0$. The proof of Theorem \ref{thm:polynomials} shows that we have ${\kappa_\varphi}\geq 1+\sum_{j=2}^d 1/k_j$.
	
	Now, let us write the Taylor expansion of $\mathrm{Re}\,\Phi$ and $\mathrm{Im}\,\Phi$ near $\mathbf 1$, but also now for a point belonging to the unit polydisc. Writing $$\Phi(z)=\sum_{j=1}^d a_j(1-z_j)+o\left(\sum_{j=1}^d |1-z_j|\right)$$ and $z=\big((1-\rho_1)e^{i\theta_1},\ldots,(1-\rho_d)e^{i\theta_d}\big)$, it is easy to get 
	\begin{eqnarray*}
		\mathrm{Re}\,\Phi(z)&=&a_1\rho_1+\cdots+a_d\rho_d+\ell_1(\theta)^{k_1}+\cdots+\ell_d(\theta)^{k_d}+o\left(\sum_{j=1}^d\left(\rho_j+\ell_j^{k_j}(\theta)\right)\right),\\
		\mathrm{Im}\,\Phi(z)&=&\tau+b_1\ell_1(\theta)+\cdots+b_d\ell_d(\theta)+o\left(\sum_{j=1}^d |\ell_j(\theta)|\right). 
	\end{eqnarray*}
	Recalling that $a_j\geq 0$ for $j=1,\,\ldots,\,d$, it is easy to conclude that there exists a neighbourhood $\mathcal U\ni\mathbf 1$ in $\overline{\DD^d}$ and $C>0$ such that, for all $z\in\mathcal U$, 
	$$|\mathrm{Im}\,\Phi(z)-\tau|^{k_1}\leq C\mathrm{Re}\,\Phi(z).$$ 
	Outside $\mathcal U$, $\mathrm{Re}\,\Phi(z)$ is bounded away from 0, and $|\textrm{Im}\,\Phi(z)-\tau|$ is here trivially majorized. Hence, the upper bound of Theorem \ref{thm:approximationnumber} follows from Theorem \ref{thm:apupperbound}. 
\end{proof}

Let us now turn to the proof of Theorem~\ref{thm:apupperbound}. The proof will be preceded by two lemmas. The first one is inspired by Lemma~3.1 in \cite{QS15a}.
\begin{lem}
	\label{lem:apupperbound1} Let $\Omega$ be a bounded domain in $\mathbb C_0$ whose boundary is a piecewise regular Jordan curve $\Gamma$, with $L_p(\Gamma)\geq 1$. Let $s_1,\,\ldots,\,s_n$ be points in $\Gamma$ such that the hyperbolic length of the curve between any two points $s_j$ and $s_{j+1}$ is equal to $L_p(\Gamma)/n$, $1\leq j\leq n$, where $s_{n+1}=s_1$. Let $B$ be the Blaschke product of degree $n$ whose zeros are precisely $s_1,\,\ldots,\,s_n$. Then, for any $s\in\Omega$, 
	\[|B(s)|\leq \exp\left(-C\frac{n}{L_p(\Gamma)}\right).\]
	\begin{proof}
		By the maximum principle, it is sufficient to prove this inequality for $s\in\Gamma$. In this case, we know that there exists some  $j \in \{1,\,\ldots,\,n\}$ such that $d(s,s_j)\leq L_p(\Gamma)/n$, from which we deduce that 
		\[d(s,s_k)\leq\frac{L_p(\Gamma)}{n}(1+|k-j|)\]
		for any $k=1,\,\ldots,\,n$. Using the link between the pseudo hyperbolic distance and the hyperbolic distance, we deduce that 
		\[|B(s)|\leq \prod_{j=1}^n\left(\frac{1-e^{-j\frac{L_p(\Gamma)}{n}}}{1+e^{-j\frac{L_p(\Gamma)}{n}}}\right).\]
		By a Riemann sum argument, this means that
		\[|B(s)|\leq\exp\left(-n\int_0^1 \ln\left(\frac{1-e^{-xL_p(\Gamma)}}{1+e^{xL_p(\Gamma)}}\right)dx\right) \leq \exp\left(-\frac{n}{L_p(\Gamma)}\int_{e^{-L_p(\Gamma)}}^1 \frac 1y\ln\left(\frac{1+y}{1-y}\right)dy\right),\]
		and we get the desired conclusion, since by assumption $L_p(\Gamma)\geq1$. 
	\end{proof}
\end{lem}

Hence, we require estimates of the hyperbolic length of some curves which are linked to the way that $\varphi$ touches the boundary. Such estimates are contained in the following result.
\begin{lem}
	\label{lem:apupperbound2} Let $\omega\geq 1$, $\sigma\in(0,1/2)$ and $C>1$. Consider $$\Omega_{\omega,\sigma,C}=\big\{s\in\CC_0\,:\, |\mathrm{Im}\,s|^\omega\leq C\mathrm{Re}(s),\ \sigma\leq\mathrm{Re}\,s\leq C\big\}.$$ Let $\Gamma_{\omega,\sigma,C}$ denote the boundary of $\Omega_{\omega,\sigma,C}$. Then \[L_p(\Gamma_{\omega,\sigma,C})\ll_{\omega,C}\begin{cases}
		\left(\frac 1\sigma\right)^{\frac{\omega-1}\omega}&\text{if }\omega>1,\\
		-\ln(\sigma)&\text{if }\omega=1.
	\end{cases}\]
	\begin{proof}
		Consider the curves
		\begin{eqnarray*}
			\Gamma_1&=&\left\{s\in\CC_0;\ \mathrm{Re}\,s=\sigma,\ |\mathrm{Im}\,s|\leq C^{1/\omega}(\mathrm{Re}\, s)^{1/\omega}\right\},\\
			\Gamma_2&=&\left\{s\in\CC_0;\ \mathrm{Re}\,s=C,\ |\mathrm{Im}\,s|\leq C^{1/\omega}(\mathrm{Re}\, s)^{1/\omega}\right\},\\
			\Gamma_3&=&\left\{s\in\CC_0;\ \sigma\leq\mathrm{Re}\, s\leq C,\ |\mathrm{Im}\,s|= C^{1/\omega}(\mathrm{Re}\, s)^{1/\omega}\right\}. 
		\end{eqnarray*}
		Clearly, $\Gamma_{\omega,\sigma,C}\subset\Gamma_1\cup\Gamma_2\cup\Gamma_3$ and it is sufficient to prove the corresponding inequalities for $\Gamma_j$, $j=1,2,3$. Firstly, $L_p(\Gamma_2)\ll_{\omega,C} 1$. Regarding $\Gamma_1$, $$L_p(\Gamma_1)=\int_{-C^{1/\omega}\sigma^{1/\omega}}^{C^{1/\omega}\sigma^{1/\omega}}\frac{dy}\sigma\ll_{\omega,C} \left(\frac1\sigma\right)^{\frac{\omega-1}\omega}$$ which is even a stronger inequality than required when $\omega=1$. Finally, $$L_p(\Gamma_3)\ll_{\omega,C}\int_\sigma^C \frac{\sqrt{1+x^{\frac 2\omega-1}}}xdx\ll_{\omega,C}\begin{cases}
			-\ln(\sigma)&\text{if }\omega=1,\\
			\left(\frac 1\sigma\right)^{\frac{\omega-1}\omega}&\text{if }\omega>1.
		\end{cases}$$ 
		The last estimate follows from inspecting the integrand near $x=0$, since $\sigma \in (0,1)$.
	\end{proof}
\end{lem}

\begin{proof}[Proof of Theorem \ref{thm:apupperbound}] 
	Let $\sigma\in (0,1)$ and $n\geq 1$. Without loss of generality, we may assume that $\Phi(\mathbf 1)=0$. Keeping the notations of Lemma \ref{lem:apupperbound2}, there exists $C>0$ such that 
	\[\varphi(\CC_0)-\frac 12\subset\big\{s\in\CC_0\,:\, 0\leq \mathrm{Re}\, s\leq \sigma\big\}\cup\Omega_{\omega_\varphi,\sigma,C}.\]
	Let $B$ be a Blaschke product of degree $n$ defined as in Lemma~\ref{lem:apupperbound1} with $\Omega_{\omega_\varphi,\sigma,C}$. Enlarging $C$ if necessary, we may always assume that $L_p(\Gamma_{\omega_\varphi,\sigma,C})\geq 1$, so that the assumptions of Lemma~\ref{lem:apupperbound1} are satisfied. The set 
	\[\Omega=\overline{\left\{\varphi(s)-\frac12 \,:\, \mathrm{Re}\,\varphi(s)\geq\frac12+\sigma\right\}}\]
	is a compact subset of $\CC_0$, and we may apply part (a) of Lemma \ref{lem:apqsmethod}. Since $\Omega\subset\Omega_{\omega_\varphi,\sigma,C}$, we obtain 
	\[\sup_{s\in\Omega}|B(s)|^2\leq \exp\left(-2C'\frac{n}{L_p(\Gamma_{\omega_\varphi,\sigma,C})}\right).\]
	Moreover, $\zeta(1+2\sigma)\ll 1/\sigma$. Finally, using Lemma \ref{lem:apcarleson1}, we obtain
	\[\|\mu_{\varphi,\Omega}\|_{\mathcal C,\mathcal H^2}\ll \sup_{\veps\in(0,\sigma),\ \tau\in\mathbb R}\frac{\mu_{\varphi,\Omega}\big(Q(\tau,\veps)\big)}{\veps}\ll \sigma^{\kappa_\varphi-1}.\]
	We will now optimize the choice of $\sigma$ with respect to $n$. When $\omega_\varphi>1$, we set 
	\[\sigma=\rho\left(\frac{\log n}n\right)^{\frac{\omega_\varphi}{\omega_\varphi-1}},\]
	where $\varrho$ is some numerical parameter to be chosen later. Then 
	\[\sup_{s\in\Omega}|B(s)|^2\zeta(1+2\sigma)\leq \exp\left(-2C'\rho^{\frac{\omega_\varphi-1}{\omega_\varphi}}\log n\right)\,\cdot\,\frac1\sigma\ll \left(\frac{\log n}n\right)^{\frac{\omega_\varphi}{\omega_\varphi-1}(\kappa_\varphi-1)},\] 
	provided $\rho>0$ is sufficiently large. When $\omega_{\varphi}\leq 1$, we set $\sigma=\exp\big(-\rho n^{-1/2}\big)$, so that
	\[\sup_{s\in\Omega}|B(s)|^2\zeta(1+2\sigma)\leq \exp\left(-\frac{C''}\rho n^{1/2}+\rho n^{1/2}\right),\] 
	and the result is proved provided $\rho>0$ is sufficiently small. 
\end{proof}
\begin{rem}
	Our method of proof also shows, provided $\varphi(\CC_0)$ is bounded and $\kappa_\varphi>1$, that 
	\[a_n(\mathcal C_\varphi)\leq \left(\frac{\log n}n\right)^{\frac{\kappa_\varphi-1}2}.\] 
	Indeed, we apply the same method with $\Omega_{\sigma,C}=\{s\in\CC_0\,:\, \sigma\leq \mathrm{Re}\,s\leq C,\ |\mathrm{Im}\,s|\leq C\}$ which satisfies $L_p(\Gamma)\ll_C \sigma^{-1}$. The rest of the proof remains unchanged. 
\end{rem}

\medskip \subsection{The Lower Bound} \label{sec:aplowerbound} \hspace{1cm} \medskip 

Let $\varphi\in\mathcal{G}$ satisfying the assumptions of Theorem~\ref{thm:approximationnumber} and let us assume that around $\mathbf 1$, $\Phi$ satisfies (\ref{eq:apboundaryregular1}) and (\ref{eq:apboundaryregular2}). Let $\nu>0$. For $\delta\in (0,1/\nu)$, we consider the sequence $S_\delta=(s_m)$, given by
\[s_m=\frac 12+\nu\delta+im\delta, \qquad\text{where }\quad 1\leq m\leq \left(\frac 1\delta\right)^{1-\frac 1{k_1}}.\]
We intend to apply part (b) of Lemma \ref{lem:apqsmethod}. We will require the construction of preimages of $S_\delta-1/2$ by $\Phi$, and the inverse function theorem will provide the solution.
\begin{lem}
	\label{lem:aplowerbound1} Let $\varphi \in \mathcal{G}$ satisfy the assumptions of Theorem~\ref{thm:approximationnumber}. Then there exist $\nu_0$, $C_1,C_2>0$ such that for all $\nu\geq \nu_0$ and every $\delta\in (0,1/\nu)$, there exists a finite sequence $Z_\delta=(Z(\alpha))$ in $\DD^d$ with
	\[Z(\alpha)=\big[(1-\rho_1(\alpha))e^{i\theta_1(\alpha)},\ldots,(1-\rho_d(\alpha))e^{i\theta_d(\alpha)}\big]\]
	such that 
	\begin{itemize}
		\item for any $\alpha\neq\beta$, we have $\sup_{j=1,\ldots,d} |\theta_j(\alpha)-\theta_j(\beta)|\geq C_1 \delta$, 
		\item for any $\alpha$ and any $j=1,\,\ldots,\,d$, we have $C_2^{-1}\delta\leq \rho_j(\alpha)\leq C_2\delta$,
		\item $\Phi(Z_\delta)=S_\delta-1/2$ and, for any $1\leq m\leq \left(\frac 1\delta\right)^{1-\frac 1{k_1}}$, the equation $\Phi(Z(\alpha))=s_m-\frac 12$ has at least $\prod_{j=2}^d \left\lfloor \left(\frac 1\delta\right)^{ 1-\frac 1{k_j}}\right\rfloor$ solutions. 
	\end{itemize}
	\begin{proof}
		We start as in the deduction of the upper bound in Theorem \ref{thm:approximationnumber} from Theorem \ref{thm:apupperbound}, writing 
		\begin{eqnarray*}
			\mathrm{Re}\,\Phi(z)&=&a_1\rho_1+\cdots+a_d\rho_d+\ell_1(\theta)^{k_1}+\cdots+\ell_d(\theta)^{k_d}+o\left(\sum_{j=1}^d\left(\rho_j+\ell_j^{k_j}(\theta)\right)\right),\\
			\mathrm{Im}\,\Phi(z)&=&\tau+b_1\ell_1(\theta)+\cdots+b_d\ell_d(\theta)+o\left(\sum_{j=1}^d |\ell_j(\theta)|\right), 
		\end{eqnarray*}
		for $z=\big((1-\rho_1)e^{i\theta_1},\,\ldots,\,(1-\rho_d)e^{i\theta_d}\big)$. To simplify the notations, we use the (linear) change of variables $u_j=\ell_j(\theta)$. We also set 
		\[\Lambda=\NN^d\cap \prod_{j=1}^d \left[1,\left(\frac 1\delta\right)^{1-\frac1{k_j}}\right]\]
		and, for $\alpha\in\Lambda$ and $j=2,\,\ldots,\,d$ we let $\rho_j(\alpha)=\delta$ and $u_j(\alpha)=\alpha_j\delta$. 
		
		Setting $m=\alpha_1$, we want to find $Z(\alpha)$ such that $\mathrm{Re}\,\Phi(Z(\alpha))=\nu\delta$ and $\textrm{Im}\,\Phi(Z(\alpha))=m\delta$. It remains to determine $\rho_1(\alpha)$ and $u_1(\alpha)$. We rewrite this system as 
		\begin{equation}
			\label{eq:aplowerbound1} \left\{ 
			\begin{array}{rcl}
				f_\alpha(\rho_1,u_1)\rho_1+g_{\alpha}(\rho_1,u_1)u_1^{k_1}&=&\nu\delta+d_\alpha \\
				h_\alpha(\rho_1,u_1)u_1&=&m\delta+e_\alpha 
			\end{array}
			\right. 
		\end{equation}
		where $f_\alpha$, $g_\alpha$ and $h_\alpha$ are smooth functions depending only on $\alpha_2,\,\ldots,\,\alpha_d$ and there exists a neighbourhood $\mathcal U\ni (0,0)$ so that for every $(\rho,u) \in \mathcal{U}$, 
		\[|f_\alpha(\rho,u)-a_1|\ll \delta, \qquad |g_\alpha(\rho,u)-1|\ll\delta \qquad \text{and} \qquad |h_\alpha(\rho,u)-1|\ll \delta.\]
		Here, the open set $\mathcal U$ and the involved constants are uniform with respect to $\alpha$, $\nu\geq 1$ and $\delta\in(0,1/\nu)$. Moreover, the real numbers $d_\alpha$ and $e_\alpha$ satisfy 
		\[d_\alpha\ll \sum_{j=2}^d \delta+\sum_{j=2}^d \left(\left(\frac1\delta\right)^{1-\frac1{k_j}}\right)^{k_j}\ll\delta$$ $$e_\alpha\ll \delta\sum_{j=2}^d \left(\frac 1\delta\right)^{1-\frac 1{k_j}}\ll \delta^{\frac1{k_1}}.\] 
		We now apply the inverse function theorem to solve the system \eqref{eq:aplowerbound1}. Provided $\nu$ is large enough, we get a solution $(\rho_1(\alpha),u_1(\alpha))$ satisfying $\sup(\rho_1(\alpha),|u_1(\alpha)|)\ll\delta^{1/k_1}$. In this case, the involved constant depends on $\nu$, but it is uniform with respect to $\alpha$ and $\delta$. 
		
		Now, a look at the first equation of (\ref{eq:aplowerbound1}) shows that we in fact have the more precise inequality $\delta\ll \rho_1(\alpha)\ll \delta$, provided $\nu$ is sufficiently large, and this is independent of $\alpha$ and $\delta\in(0,1)$. Looking now at the second equation of \eqref{eq:aplowerbound1}, if $\alpha\neq \beta\in\Lambda$ satisfy $\alpha_j=\beta_j$ for $j\geq 2$, so that $e_\alpha=e_\beta$ and $h_\alpha=h_\beta$, then $|u_1(\alpha)-u_1(\beta)|\gg \delta$. 
	
		Hence, we have obtained $\prod_{j=2}^d \left\lfloor \left(\frac 1\delta\right)^{ 1-\frac 1{k_j}}\right\rfloor$ solutions to the equation $\Phi(Z(\alpha))=s_m$, and they satisfy the conclusions of Lemma \ref{lem:aplowerbound1} since the inequalities on $u_j(\alpha)$ are also valid for $\theta_j(\alpha)$ up to a constant depending only of $\Phi$. 
	\end{proof}
\end{lem}

\begin{proof}[Final part in the proof of the lower bound of Theorem \ref{thm:approximationnumber}]
	We apply Lemma \ref{lem:apqsmethod} to $S_\delta$ and $Z_\delta$ given by the previous lemma, for 
	\[\delta = \left(\frac{1}{n}\right)^{\frac{k_1}{k_1-1}},\]
	so that $S_\delta$ has cardinal number equal to $n$. Since $M_{\mathcal H^2}(S_\delta)\ll 1$ and $\left\|\mu_{Z,H^2(\TT^d)}\right\|_{\mathcal C,H^2(\TT^d)}\ll 1$ by Lemma~\ref{lem:apinterpolation1} and Lemma~\ref{lem:apinterpolation2}, it remains to estimate the sum $N_\Phi(s_m;Z)\zeta(2\mathrm{Re}\,s_m)$ for any $m$. Using the fact that $\rho_j(\alpha)\gg\delta$ for any $j=1,\,\ldots,\,d$ and any $\alpha$, we obtain
	\[N_\Phi(s_m;Z)\zeta(2\mathrm{Re}\,s_m)\gg\left(\frac 1\delta\right)^{\sum_{j=2}^d\left(1-\frac 1{k_j}\right)}\,\cdot\, \delta^{d}\,\cdot\, \delta^{-1} \gg \delta^{\sum_{j=2}^d \frac 1{k_j}} \gg \left(\frac 1n\right)^{\left(\sum_{j=2}^d \frac1{k_j}\right)\times\frac{k_1}{k_1-1}},\]
	and we are done.
\end{proof}

\begin{rem}
	We may modify the proof of Theorem \ref{thm:approximationnumber} so that we do not assume that there exists a unique $w\in\TT^d$ such that $\mathrm{Re}\,\Phi(w)=0$. Suppose that $\varphi$ is boundary regular at any $w\in\TT^d$ such that $\mathrm{Re}\,\Phi(w)=0$. Define now the \emph{compactness index} of $\varphi$ as the real number $$\eta_\varphi(s)=\inf\big\{\eta_{\varphi,w}\,:\,\mathrm{Re}\, \Phi(w)=0\big\}.$$ It should be observed that this infimum is in fact a minimum. Indeed, our assumptions imply that the points $w\in\TT^d$ such that $\mathrm{Re}\,\Phi(w)=0$ are isolated. Theorem \ref{thm:approximationnumber} remains true with this new definition of $\eta_\varphi$. 
\end{rem}

\bibliographystyle{amsplain} 
\bibliography{nonlin}

\providecommand{\bysame}{\leavevmode\hbox to3em{\hrulefill}\thinspace}
\providecommand{\MR}{\relax\ifhmode\unskip\space\fi MR }
\providecommand{\MRhref}[2]{%
  \href{http://www.ams.org/mathscinet-getitem?mr=#1}{#2}
}
\providecommand{\href}[2]{#2}
\begin{thebibliography}{10}

\bibitem{BAYMONAT}
F.~Bayart, \emph{{Hardy spaces of Dirichlet series and their composition
  operators}}, Monatsh. Math. \textbf{136} (2002), 203--236.

\bibitem{Ba03}
\bysame, \emph{Compact composition operators on a {H}ilbert space of
  {D}irichlet series}, Illinois J. Math. \textbf{47} (2003), no.~3, 725--743.

\bibitem{Xcas}
F.~Bayart and O.~F. Brevig, \emph{{Xcas} file for {L}emma~11},
  \url{http://bayart.perso.math.cnrs.fr/nonlin.xws}.

\bibitem{BCL87}
B.~Berndtsson, A.~Chang, and K-C. Lin, \emph{{Interpolating sequences in the
  polydisc}}, Trans. of the AMS \textbf{302} (1987), 161--169.

\bibitem{BG92}
J.~W. Bruce and P.~J. Giblin, \emph{Curves and {S}ingularities: {A} geometrical
  introduction to singularity theory}, Cambridge University Press, 1992.

\bibitem{Ch79}
A.~Chang, \emph{Carleson measure on the bi-disc}, Ann. of Math. \textbf{109}
  (1979), 613--620.

\bibitem{FQ04}
C.~Finet and H.~Queff{\'e}lec, \emph{{Numerical range of composition operators
  on a Hilbert space of Dirichlet series}}, Linear Algebra Appl. \textbf{377}
  (2004), 1--10.

\bibitem{FQV04}
C.~Finet, H.~Queff{\'e}lec, and A.~Volberg, \emph{Compactness of composition
  operators on a {H}ilbert space of {D}irichlet series}, J. Funct. Anal.
  \textbf{211} (2004), no.~2, 271--287.

\bibitem{GH99}
J.~Gordon and H.~Hedenmalm, \emph{{The composition operators on the space of
  Dirichlet series with square summable coefficients}}, Michigan Math. J.
  \textbf{46} (1999), no.~2, 313--329.

\bibitem{hardywright}
G.~H. Hardy and E.~M. Wright, \emph{An introduction to the theory of numbers},
  Oxford University Press, 1979.

\bibitem{QS14}
H.~Queff{\'e}lec and K.~Seip, \emph{Approximation numbers of composition
  operators on the {$H^2$} space of {D}irichlet series}, J. Funct. Anal.
  \textbf{268} (2015), 1612--1648.

\bibitem{QS15a}
H.~Queff\'elec and K.~Seip, \emph{{Decay rates for approximation numbers of
  composition operators}}, J. Anal. Math. \textbf{125} (2015), 371--399.

\end{thebibliography}

\end{document}